\documentclass{article}

\usepackage{amsmath,amssymb,amsthm}

\usepackage{a4wide}
\usepackage{subfig}
\usepackage{graphicx}

\usepackage{url}

\newtheorem{theorem}{Theorem}[section]
\newtheorem{proposition}{Proposition}[section]
\newtheorem{definition}{Definition}[section]
\newtheorem{remark}{Remark}[section]
\newtheorem{example}{Example}[section]

\begin{document}

\title{Optimal control for a tuberculosis model\\
with reinfection and post-exposure interventions\thanks{This is a preprint 
of a paper whose final and definite form will appear in 
\emph{Mathematical Biosciences}. Paper submitted 9-Dec-2011; 
revised 9-Jun-2012, 13-Jan-2013, and 7-May-2013; 
accepted for publication 9-May-2013.}}

\author{Cristiana J. Silva\\ \texttt{cjoaosilva@ua.pt}
\and Delfim F. M. Torres\\ \texttt{delfim@ua.pt}}

\date{CIDMA -- Center for Research and Development in Mathematics and Applications,\\
Department of Mathematics, University of Aveiro, 3810-193 Aveiro, Portugal}

\maketitle


\begin{abstract}
We apply optimal control theory to a tuberculosis model given
by a system of ordinary differential equations. Optimal control
strategies are proposed to minimize the cost of interventions,
considering reinfection and post-exposure interventions.
They depend on the parameters of the model and reduce effectively
the number of active infectious and persistent latent individuals.
The time that the optimal controls are at the upper bound increase
with the transmission coefficient. A general explicit expression
for the basic reproduction number is obtained and its sensitivity
with respect to the model parameters is discussed. Numerical results
show the usefulness of the optimization strategies.
\end{abstract}




\section{Introduction}

\emph{Mycobacterium tuberculosis} is the cause of most occurrences of tuberculosis (TB)
and is usually acquired via airborne infection from someone who has active TB.
It typically affects the lungs (pulmonary TB) but can affect other sites as well
(extrapulmonary TB). Only approximately 10\% of people infected with
\emph{mycobacterium tuberculosis} develop active TB disease. Therefore,
approximately 90\% of people infected remain latent. Latent infected TB people
are asymptomatic and do not transmit TB, but may progress to active TB through
either endogenous reactivation or exogenous reinfection \cite{Small_Fuj_2001,Styblo_1978}.

Without treatment, mortality rates are high, but the anti-TB drugs developed since 1940
dramatically reduce mortality rates (in clinical cases, cure rates of 90\% have
been documented) \cite{WHO_2011}. However, TB remains a major health problem.
In 2010 there were an estimated 8.5 to 9.2 million cases and 1.2 to 1.5 million deaths.
TB is the second leading cause of death from an infectious disease worldwide
after HIV \cite{WHO_2011}.

One can distinguish three types of TB treatment: vaccination to prevent infection;
treatment to cure active TB; treatment of latent TB to prevent endogenous
reactivation \cite{Gomes_etall_2007}.
The treatment of active infectious individuals can have different
timings \cite{Kruk_etall_2008}. Here we consider treatment with
the duration of six months. In these treatments one of the difficulties
to their success is to make sure that the patients complete the treatment.
Indeed, after two months, patients no longer have symptoms of the disease and feel healed,
and many of them stop taking the medicines. When the treatment is not concluded,
the patients are not cured and reactivation can occur and/or the patients
may develop resistent TB. One way to prevent patients of not completing
the treatment is based on supervision and patient support. In fact, this is one
of the measures proposed by the Direct Observation Therapy (DOT)
of World Health Organization (WHO) \cite{WHO_treatTB_2010}.
One example of treatment supervision consists in recording each dose of anti-TB drugs
on the patients treatment card \cite{WHO_treatTB_2010}. These measures
are very expensive since the patients need to stay longer in the hospital
or specialized people are to be payed to supervise patients till
they finish their treatment. On the other hand, it is recognized
that the treatment of latent TB individuals reduces the chances of reactivation,
even if it is still unknown how treatment influences reinfection \cite{Gomes_etall_2007}.

Optimal control is a branch of mathematics developed
to find optimal ways to control a dynamic system
\cite{Cesari_1983,Fleming_Rishel_1975,Pontryagin_et_all_1962}.
While the usefulness of optimal control theory in epidemiology is nowadays well recognized
\cite{livro_Lenhart_2007,Rodrigues_Monteiro_Torres_2010,Rodrigues_Monteiro_Torres_Zinober_2011},
results in tuberculosis are scarce \cite{SLenhart_2002}.
Recently, different optimal control problems applied to TB have been proposed and analyzed
\cite{Bowong_2010,Emvudu_et_all,Haffat_et_all}.
The first paper appeared in 2002 \cite{SLenhart_2002}, and considers a mathematical model
for TB based on \cite{Castillo_Chavez_1997} with two classes of infected and latent individuals
(infected with typical TB and with resistant strain TB) where the aim is to reduce the number
of infected and latent individuals with resistant TB. In \cite{Emvudu_et_all} the model considers
the existence of a class called \emph{the lost to follow up individuals} and they propose optimal
control strategies for the reduction of the number of individuals in this class. In \cite{Haffat_et_all}
the authors adapt a  model from \cite{Castillo_Chavez_2000} where exogenous reinfection is considered
and wish to minimize the number of infectious individuals. In \cite{Bowong_2010} a TB model
that incorporates exogenous reinfection, chemoprophylaxis of latently infected individuals
and treatment of infections is proposed. Optimal control strategies based on chemoprophylaxis
of latently infected individuals and treatment of infectious individuals are analyzed
for the reduction of the number of active infected individuals.
Our aim is to study optimal strategies for the minimization of the number
of active TB infectious and persistent latent individuals, taking into account
the cost of the measures for the treatments of these individuals.
For that, we study the mathematical model for TB dynamics presented
in \cite{Gomes_etall_2007}, where reinfection and post-exposure interventions
are considered. The importance of considering reinfection and post-exposure
interventions is justified in \cite{Bandera_2001,Caminero_2001,Gomes_etall_2007,van_Rie_1999}.
In Section~\ref{sec:TBmodel} we modify the model of \cite{Gomes_etall_2007}
adding two controls $u_1(t)$ and $u_2(t)$, which are functions of time $t$,
and two real positive parameters, $\epsilon_1$ and $\epsilon_2$.
We also explain the meaning of these and the other parameters of the TB model.
A sensitivity analysis for the basic reproduction number
is given in Section~\ref{sec:sensitivity}.
In Section~\ref{sec:opt:cont:prob} we formulate the optimal control problem.
We prove that the problem has an unique solution, and finally we apply to it
the celebrated Pontryagin Maximum Principle \cite{Pontryagin_et_all_1962}.
In Section~\ref{sec:num:results} we propose optimal control strategies,
obtained by numerical simulations, considering several variations
of some of the parameters of the TB model.
We end with Section~\ref{sec:conc} of conclusion.


\section{TB model with controls}
\label{sec:TBmodel}

We study the mathematical model from \cite{Gomes_etall_2007}
where reinfection and post-exposure interventions are considered.
We add to that model two control functions $u_1(\cdot)$ and $u_2(\cdot)$
and two real positive parameters $\epsilon_1$ and $\epsilon_2$.
The resulting model is given by the following system
of nonlinear ordinary differential equations:
\begin{equation}
\label{modelGab_controls}
\begin{cases}
\dot{S}(t) = \mu N - \frac{\beta}{N} I(t) S(t) - \mu S(t)\\
\dot{L_1}(t) = \frac{\beta}{N} I(t)\left( S(t) + \sigma L_2(t)
+ \sigma_R R(t)\right) - (\delta + \tau_1 + \mu)L_1(t)\\
\dot{I}(t) = \phi \delta L_1(t) + \omega L_2(t) + \omega_R R(t)
- (\tau_0 + \epsilon_1 u_1(t) + \mu) I(t)\\
\dot{L_2}(t) = (1 - \phi) \delta L_1(t) - \sigma \frac{\beta}{N} I(t) L_2(t)
- (\omega + \epsilon_2 u_2(t) + \tau_2 + \mu)L_2(t)\\
\dot{R}(t) = (\tau_0 + \epsilon_1 u_1(t) )I(t) +  \tau_1 L_1(t)
+ \left(\tau_2 +\epsilon_2 u_2(t)\right) L_2(t)
- \sigma_R \frac{\beta}{N} I(t) R(t) - \left(\omega_R + \mu\right) R(t) .
\end{cases}
\end{equation}
The population is divided into five categories (\textrm{i.e.},
control system \eqref{modelGab_controls} has five state variables):
susceptible ($S$); early latent ($L_1$), \textrm{i.e.}, individuals recently infected
(less than two years) but not infectious; infected ($I$), \textrm{i.e.},
individuals who have active TB and are infectious;
persistent latent ($L_2$), \textrm{i.e.}, individuals who were infected and remain latent;
and recovered ($R$), \textrm{i.e.}, individuals who were previously infected and treated.
The control $u_1$ represents the effort that prevents the failure of treatment
in active TB infectious individuals $I$, \textrm{e.g.}, supervising the patients,
helping them to take the TB medications regularly and to complete the TB treatment.
The control $u_2$ represents the fraction of persistent latent individuals $L_2$
that is identified and put under treatment.
The parameters $\epsilon_i$, $\epsilon_i \in (0, 1)$, $i=1, 2$,
measure the effectiveness of the controls $u_i$, $i=1, 2$, respectively, \textrm{i.e.},
these parameters measure the efficacy of treatment interventions for active and persistent
latent TB individuals, respectively.

Following \cite{Gomes_etall_2007}, we assume that at birth all individuals are equally susceptible
and differentiate as they experience infection and respective therapy.
Moreover, the total population, $N$, with $N = S + L_1 + I + L_2 + R$, is assumed to be constant,
\textrm{i.e.}, the rate of birth and death, $\mu$, are equal
(corresponding to a mean life time of 70 years \cite{Gomes_etall_2007})
and there are no disease-related deaths. The assumption that
the total population $N$ is constant, allows to reduce
the control system \eqref{modelGab_controls} from five to four state variables.
We decided to maintain the TB model in form \eqref{modelGab_controls},
using relation $S(t) + L_1(t) + I(t) + L_2(t) + R(t) = N$ as a test
to confirm the numerical results. The proportion of population change, in each category,
is described by system \eqref{modelGab_controls}. The initial value of each category,
$S(0)$, $L_1(0)$, $I(0)$, $L_2(0)$ and $R(0)$, are given in Table~\ref{parameters}
and are based on \cite{SLenhart_2002}.

The values of the rates $\delta$, $\phi$, $\omega$, $\omega_R$, $\sigma$ and
$\tau_0$ are taken from \cite{Gomes_etall_2007} and the references cited therein
(see Table~\ref{parameters} for the values of the parameters). The parameter $\delta$
denotes the rate at which individuals leave $L_1$ compartment; $\phi$ is the proportion
of individuals going to compartment $I$; $\omega$ is the rate of endogenous reactivation
for persistent latent infections (untreated latent infections); $\omega_R$ is the rate
of endogenous reactivation for treated individuals
(for those who have undergone a therapeutic intervention).
The parameter $\sigma$ is the factor that reduces the risk of infection,
as a result of acquired immunity to a previous infection,
for persistent latent individuals, \textrm{i.e.}, this factor affects
the rate of exogenous reinfection of untreated individuals;
while $\sigma_R$ represents the same parameter factor but for treated patients.
In our simulations we consider the case where the susceptibility to reinfection
of treated individuals equals that of latents: $\sigma_R = \sigma$.

The parameter $\tau_0$ is the rate of recovery under treatment of active TB
(assuming an average duration of infectiousness of six months).
The parameters $\tau_1$ and $\tau_2$ apply to latent individuals $L_1$ and $L_2$,
respectively, and are the rates at which chemotherapy or a post-exposure vacine is applied.
In \cite{Gomes_etall_2007} different values for these rates are considered:
the case where no treatment of latent infections occur ($\tau_1 = \tau_2 = 0$);
the case where there is an immediate treatment of persistent latent infections ($\tau_2 \to \infty$);
or there is a moderate treatment of persistent latent infections ($\tau_2 \in [0.1,1]$).
The first and second cases are not interesting from the optimal control point of view.
In our paper we consider, without loss of generality, that the rate of recovery of early latent
individuals under post-exposure interventions is equal to the rate of recovery under treatment of active TB,
$\tau_1 = 2 \, yr^{-1}$, and greater than the rate of recovery of persistent latent individuals
under post-exposure interventions, $\tau_2 = 1 \, yr^{-1}$.

It is assumed that the rate of infection of susceptible individuals is proportional
to the number of infectious individuals and the constant of proportionality is $\beta$,
which is the transmission coefficient. The basic reproduction number $R_0$,
for system \eqref{modelGab_controls} in the absence of controls,
i.e., in the case $u_1 = u_2 = 0$,
is proportional to the transmission coefficient $\beta$
(see \cite{Gomes_etall_2007}) and is given by
\begin{equation}
\label{r0:Gab}
R_0 = {\frac {\delta\,
\left( \omega + \phi\, \mu \right)
\left( {\it \omega_R}+\mu \right)}{\left( {\it \omega_R}+{\it \tau_0} + \mu\right)
\left( \delta + \mu\right) \left( \omega + \mu\right) }}
\, \frac{\beta}{\mu}.
\end{equation}
To see that the controls can be used to reduce $R_0$, one just needs
to follow the same procedure used to obtain \eqref{r0:Gab} in \cite{Gomes_etall_2007},
for the general situation where the controls $u_1$ and $u_2$ are present:
\begin{equation}
\label{r0}
R_0(u_1,u_2) =
{\frac {\delta \left[
\left( \omega + \phi\,\mu \right)
\left( {\it \omega_R}+\mu \right)
+ \left({\it \omega_R} + \phi \, \mu \right)\,{\it \epsilon_2} \,{\it u_2}\right]}{
\left( {\it \omega_R}+{\it \tau_0} + \mu +{\it \epsilon_1}\,{\it u_1} \right)
\left( \delta + \mu\right)
\left( \omega + \mu +{\it \epsilon_2}\,{\it u_2}\right) }}
\, \frac{\beta}{\mu}.
\end{equation}
For $u_1 = u_2 = 0$ \eqref{r0} reduces to \eqref{r0:Gab}, i.e., $R_0(0,0) = R_0$.
It should be noted, however, that \eqref{r0:Gab} and \eqref{r0}
are deduced under the assumption that $\tau_1 = \tau_2 = 0$, which
in our investigation is not true: as already mentioned, we consider
$\tau_1 = 2$ and $\tau_2 = 1$ (see also Table~\ref{parameters}).
Therefore, we begin by proving a general formula for $R_0(u_1,u_2)$.

\begin{proposition}
\label{prop:r0}
The basic reproduction number $R_0(u_1,u_2)$ for system \eqref{modelGab_controls} is given by
\begin{equation}
\label{eq:r0:prop}
R_0(u_1,u_2) =
{\frac{{\it \omega_R}\,
\left(  \omega+{\it \tau_2} + \mu + {\it \epsilon_2}\,{\it u_2} \right) \,{\it \tau_1}
+ \delta \left[
\left( \omega + \phi\,\mu \right)
\left( {\it \omega_R}+\mu \right)
+ \left({\it \omega_R} + \phi \, \mu \right)\,\left({\it \tau_2} + {\it \epsilon_2} \,{\it u_2}\right)
\right]}{\left(
{\it \omega_R}+{\it \tau_0} + \mu
+{\it \epsilon_1}\,{\it u_1} \right)
\left(\delta + {\it \tau_1} + \mu \right)
\left(\omega +{\it \tau_2} + \mu + {\it \epsilon_2}\,{\it u_2}\right)}}
\, \frac{\beta}{\mu}.
\end{equation}
\end{proposition}

\begin{proof}
The system \eqref{modelGab_controls} has only one DFE (disease free equilibrium):
\begin{equation}
\label{eq:DFE}
(S,L_1,I,L_2,R) = (N, 0, 0, 0, 0, 0).
\end{equation}
We calculate the basic reproduction number $R_0(u_1, u_2)$ using the approach
explained with details, \textrm{e.g.}, in \cite{Rodrigues_Monteiro_Torres_Zinober_2011}.
For that we write the right-hand side of system \eqref{modelGab_controls}
as $\mathcal{F} - \mathcal{V}$ with
$$
\mathcal{F} =
\left(
\begin{array}{c}
0 \\
{\frac {\beta\,{\it I}\, \left( S+\sigma\,{\it L_2}+{\it \sigma_R}\,R\right)}{N}}\\
0 \\
0 \\
0 \\
\end{array}
\right)
$$
and
$$
\mathcal{V} =
\left(
\begin{array}{c}
-\mu\,N+{\frac {\beta\,{\it I}\,S}{N}}+\mu\,S\\
\left( \delta+{\it \tau_1}+\mu \right) {\it L_1}\\
-\phi\,\delta\,{\it L_1}-\omega\,{\it L_2}-{\it \omega_R}\,R+ \left( {\it
\tau_0}+{\it \epsilon_1}\,{\it u_1}+\mu \right) {\it I}\\
- \left( 1-\phi \right) \delta\,{\it L_1}+{\frac {\beta\,{\it I}\,
\sigma\,{\it L_2}}{N}}+ \left( \omega+{\it \epsilon_2}\,{\it u_2}+{\it
\tau_2}+\mu \right) {\it L_2} \\
-{\frac {{\it I}\,N{\it \tau_0}+{\it I}\,N{\it \epsilon_1}\,{\it u_1}+{
\it \tau_1}\,{\it L_1}\,N+{\it L_2}\,N{\it \tau_2}+{\it L_2}\,N{\it \epsilon_2}
\,{\it u_2}-{\it \sigma_R}\,\beta\,{\it I}\,R-R N{\it \omega_R}-R N\mu}{N}}
\end{array}
\right).
$$
Then we consider the Jacobian matrices associated with $\mathcal{F}$
and $\mathcal{V}$:
$$
J_{\mathcal{F}} =
 \left[ \begin {array}{ccccc} 0&0&0&0&0\\ \noalign{\medskip}{\frac {
\beta\,{\it I}}{N}}&0&{\frac {\beta\, \left( S+\sigma\,{\it L2}+{\it
\sigma_R}\,R \right) }{N}}&{\frac {\beta\,{\it I}\,\sigma}{N}}&{\frac {
\beta\,{\it I}\,{\it \sigma_R}}{N}}\\ \noalign{\medskip}0&0&0&0&0
\\ \noalign{\medskip}0&0&0&0&0\\ \noalign{\medskip}0&0&0&0&0
\end {array} \right],
$$
$$
J_{\mathcal{V}} =
 \left[ \begin {array}{ccccc} {\frac {\beta\,{\it I}}{N}}+\mu&0&{
\frac {\beta\,S}{N}}&0&0\\ \noalign{\medskip}0&\delta+{\it \tau_1}+\mu&0
&0&0\\ \noalign{\medskip}0&-\phi\,\delta&{\it \tau_0}+{\it \epsilon_1}\,{
\it u_1}+\mu&-\omega&-{\it \omega_R}\\ \noalign{\medskip}0&- \left( 1-
\phi \right) \delta&{\frac {\sigma\,\beta\,{\it L2}}{N}}&{\frac {\beta
\,{\it I}\,\sigma}{N}}+\omega+{\it \epsilon_2}\,{\it u_2}+{\it \tau_2}+\mu
&0\\ \noalign{\medskip}0&-{\it \tau_1}&-{\it \tau_0}-{\it \epsilon_1}\,{\it
u_1}+{\frac {{\it \sigma_R}\,\beta\,R}{N}}&-{\it \tau_2}-{\it \epsilon_2}\,{
\it u_2}&{\frac {\beta\,{\it I}\,{\it \sigma_R}}{N}}+{\it \omega_R}+\mu
\end {array} \right].
$$
The basic reproduction number \eqref{eq:r0:prop}
is obtained as the spectral radius of the matrix
$J_{\mathcal{F}} \times (J_{\mathcal{V}})^{-1}$ at the disease-free
equilibrium \eqref{eq:DFE}.
\end{proof}

\begin{remark}
In the particular case $u_1 = u_2 = 0$, Proposition~\ref{prop:r0}
improves the result of \cite{Gomes_etall_2007}
for arbitrary values of $\tau_1$ and $\tau_2$.
\end{remark}

It is easy to see that the control $u_1$ has a very important role
in the reduction of $R_0$. In particular, using the control $u_1$
we can always diminish the basic reproduction number:
$R_0(u_1,0) \le R_0(0,0)$ (see Proposition~\ref{prop:role:u1:red:r0}).
The endemic threshold $ET$ at $R_0 = 1$ indicates the minimal transmission
potential that sustains endemic disease, \textrm{i.e.}, when $R_0 < 1$
the disease will die out and for $R_0 > 1$ the disease may become endemic.
In Section~\ref{sec:num:results} we take increasing values for $\beta$ with $R_0 > 1$.

\begin{example}
If $\beta = 55$, $\delta = 12$, $\omega = 0.0002$,
$\omega_R = 0.00002$, $\mu = 1/70$, $\phi = 0.05$,
$\epsilon_1 = 0.5$, $\epsilon_2 = 0.5$,
$\tau_0 = 2$, $\tau_1 = 2$ and $\tau_2 = 1$,
then the basic reproduction number $R_0$ without controls
is greater than one ($R(0, 0) > 1.21$) and with
$u_1 = u_2 = 1$ one is able to change $R_0$ to
a desirable situation: the basic reproduction number
is then less than one ($R(1, 1) < 0.97$).
\end{example}

The total simulation duration, $T$, is fixed.
Following \cite[Sec.~5.5]{Styblo_1991}, the risk of developing disease after
infection is much higher in the first five years following infection,
and decline exponentially after that. For this reason
we take $T = 5$, in years.


\section{Sensitivity of the basic reproduction number}
\label{sec:sensitivity}

The sensitivity of the basic reproduction number \eqref{eq:r0:prop}
is an important issue because it determines the model robustness to
parameter values. Two parameters have, in different directions,
a high impact on $R_0$: $u_1$ decreases $R_0$ and $\beta$ increases $R_0$.
The sensitivity of $R_0$ with respect to $u_1$
is given by Proposition~\ref{prop:role:u1:red:r0}.

\begin{proposition}
\label{prop:role:u1:red:r0}
The inequality
$$
R_0(u_1,u_2) \le R_0(0,u_2)
$$
holds for any value of the system parameters.
\end{proposition}

\begin{proof}
A direct calculation from \eqref{eq:r0:prop} shows that
$$
\frac{\partial R_0(u_1,u_2)}{\partial u_1} \le 0
$$
for any value of the parameters. Thus,
the basic reproduction number decreases with $u_1$.
\end{proof}

\begin{remark}
The analogous to Proposition~\ref{prop:role:u1:red:r0}
for $u_2$ is not true: the basic reproduction number can decrease or increase
with $u_2$ depending on the system parameters.
\end{remark}

Let us examine now the sensitivity of $R_0$
with respect to $\beta$.

\begin{proposition}
\label{prop:role:beta:red:r0}
The basic reproduction number $R_0(u_1,u_2)$
increases with $\beta$.
\end{proposition}

\begin{proof}
The results follows immediately from the fact that
$$
\frac{\partial R_0(u_1,u_2)}{\partial \beta} > 0
$$
for any value of the parameters.
\end{proof}

The sensitivity of a variable (in our case of interest, $R_0$) with respect to model parameters
is sometimes measured by the so called \emph{sensitivity index}.

\begin{definition}[cf. \cite{Chitnis,Kong}]
\label{def:sense}
The normalized forward sensitivity index of a variable $\upsilon$ that depends
differentiably on a parameter $p$ is defined by
\begin{equation}
\label{eq:def:sense}
\Upsilon_{p}^{\upsilon} := \frac{\partial \upsilon}{\partial p} \times \frac{p}{|\upsilon|}.
\end{equation}
\end{definition}

Note that to the most sensitive parameter $p$ it corresponds a normalized forward sensitivity index
of one or minus one: $\Upsilon_{p}^{\upsilon} = \pm 1$. If $\Upsilon_{p}^{\upsilon} = + 1$,
an increase (decrease) of $p$ by $x \%$ increases (decreases) $\upsilon$ by $x \%$;
if $\Upsilon_{p}^{\upsilon} = - 1$, an increase (decrease) of $p$ by $x \%$ decreases (increases)
$\upsilon$ by $x \%$. From Definition~\ref{def:sense} and Proposition~\ref{prop:r0}, it is easy
to derive the normalized forward sensitivity index of $R_0$ with respect to $\beta$ and $u_1$.

\begin{proposition}
The normalized forward sensitivity index of $R_0$ with respect to $\beta$ is $1$,
that is, $\Upsilon_{\beta}^{R_0} = 1$, while the sensitivity of $R_0$
with respect to $u_1$ is given by
$$
\Upsilon_{u_1}^{R_0}
= -{\frac {{\it \epsilon_1}\,{\it u_1}}{{\it \omega_R}
+{\it \tau_0} + \mu+{\it \epsilon_1}\,{\it u_1}}}.
$$
\end{proposition}

\begin{proof}
It follows from Proposition~\ref{prop:r0} and \eqref{eq:def:sense}:
\begin{equation*}
\begin{split}
\Upsilon_{\beta}^{R_0} &= \frac{\partial R_0}{\partial \beta} \times \frac{\beta}{|R_0|} = 1,\\
\Upsilon_{u_1}^{R_0} &= \frac{\partial R_0}{\partial u_1} \times \frac{u_1}{|R_0|} =
-{\frac {{\it \epsilon_1}\,{\it u_1}}{{\it \omega_R}+{\it \tau_0} + \mu+{\it \epsilon_1}\,{\it u_1}}}.
\end{split}
\end{equation*}
\end{proof}


\section{Optimal control problem}
\label{sec:opt:cont:prob}

In this section we present the optimal control problem,
describing our goal and the restrictions of the epidemic.
The aim is to find the optimal values $u_1^*$ and $u_2^*$ of the controls
$u_1$ and $u_2$, such that the associated state trajectories
$S^*, L_1^*, I^*, L_2^*, R^*$ are solution of the system
\eqref{modelGab_controls} in the time interval $[0, T]$
with initial conditions $S^*(0), L_1^*(0), I^*(0), L_2^*(0), R^*(0)$,
and minimize the objective functional. Here the objective functional considers
the number of active TB infectious individuals $I$, the number of persistent latent individuals $L_2$,
and the implementation cost of the strategies associated to the controls $u_i$, $i=1, 2$.
The controls are bounded between $0$ and $1$. When the controls vanish, no extra measures
are implemented for the reduction of $I$ and $L_2$; when the controls take the maximum value $1$,
the magnitude of the implemented measures, associated to $u_1$ and $u_2$, take the value
of the effectiveness of the controls, $\epsilon_1$ and $\epsilon_2$, respectively.

We consider the state system \eqref{modelGab_controls}
of ordinary differential equations in $\mathbb{R}^5$
with the set of admissible control functions given by
\begin{equation*}
\Omega = \left\{ (u_1(\cdot), u_2(\cdot)) \in (L^{\infty}(0, T))^2 \,
| \,  0 \leq u_1 (t), u_2(t) \leq 1 ,  \, \forall \, t \in [0, T] \, \right\} .
\end{equation*}
The objective functional is given by
\begin{equation}
\label{costfunction}
J(u_1(\cdot), u_2(\cdot)) = \int_0^{T} \left[ I(t) + L_2(t)
+ \frac{W_1}{2}u_1^2(t) + \frac{W_2}{2}u_2^2(t) \right] dt \, ,
\end{equation}
where the constants $W_1$ and $W_2$ are a measure
of the relative cost of the interventions
associated to the controls $u_1$, $u_2$, respectively
(see Section~\ref{sec:num:results} for further details).
We consider the optimal control problem of determining
$\left(S^*(\cdot), L_1^*(\cdot), I^*(\cdot), L_2^*(\cdot), R^*(\cdot)\right)$,
associated to an admissible control pair
$\left(u_1^*(\cdot), u_2^*(\cdot) \right) \in \Omega$ on the time interval $[0, T]$,
satisfying \eqref{modelGab_controls},
the initial conditions $S(0)$, $L_1(0)$, $I(0)$, $L_2(0)$ and $R(0)$
(see Table~\ref{parameters}) and
minimizing the cost function \eqref{costfunction}, \textrm{i.e.},
\begin{equation}
\label{mincostfunct}
J(u_1^*(\cdot), u_2^*(\cdot))
= \min_{\Omega} J(u_1(\cdot), u_2(\cdot)) \, .
\end{equation}

The existence of optimal controls $\left(u_1^*(\cdot), u_2^*(\cdot)\right)$
comes from the convexity of the cost function \eqref{costfunction}
with respect to the controls and the regularity of the system
\eqref{modelGab_controls} (see, \textrm{e.g.},
\cite{Cesari_1983,Fleming_Rishel_1975} for existence results of optimal solutions).

According to the Pontryagin Maximum Principle \cite{Pontryagin_et_all_1962},
if $(u_1^*(\cdot), u_2^*(\cdot)) \in \Omega$ is optimal for the problem
\eqref{modelGab_controls}, \eqref{mincostfunct} with the initial conditions given
in Table~\ref{parameters} and fixed final time $T$, then there exists
a nontrivial absolutely continuous mapping $\lambda : [0, T] \to \mathbb{R}^5$,
$\lambda(t) = \left(\lambda_1(t), \lambda_2(t),
\lambda_3(t), \lambda_4(t), \lambda_5(t)\right)$,
called \emph{adjoint vector}, such that
\begin{equation*}
\dot{S} = \frac{\partial H}{\partial \lambda_1} \, , \quad
\dot{L}_1= \frac{\partial H}{\partial \lambda_2} \, , \quad
\dot{I}= \frac{\partial H}{\partial \lambda_3} \, , \quad
\dot{L}_2 = \frac{\partial H}{\partial \lambda_4} \, , \quad
\dot{R} = \frac{\partial H}{\partial \lambda_5}
\end{equation*}
and
\begin{equation}
\label{adjsystemPMP}
\dot{\lambda}_1 = -\frac{\partial H}{\partial S} \, , \quad
\dot{\lambda}_2 = -\frac{\partial H}{\partial L_1} \, , \quad
\dot{\lambda}_3 = -\frac{\partial H}{\partial I} \, , \quad
\dot{\lambda}_4 = -\frac{\partial H}{\partial L_2} \, , \quad
\dot{\lambda}_5 = -\frac{\partial H}{\partial R} \, ,
\end{equation}
where function $H$ defined by
\begin{equation*}
\begin{split}
H&= H(S(t), L_1(t), I(t), L_2(t), R(t), \lambda(t), u_1(t), u_2(t)) \\
&=I(t) + L_2(t)  + \frac{W_1}{2}u_1^2(t) + \frac{W_2}{2}u_2^2(t)\\
&\, \, + \lambda_1(t) \left(\mu N - \frac{\beta}{N} I(t) S(t) - \mu S(t) \right)\\
&\, \, + \lambda_2(t) \left( \frac{\beta}{N} I(t)\left( S(t)
+ \sigma L_2(t) + \sigma_R R(t)\right) - (\delta + \tau_1 + \mu)L_1(t) \right)\\
&\, \, + \lambda_3(t) \left(\phi \delta L_1(t) + \omega L_2(t)
+ \omega_R R(t) - (\tau_0 + \epsilon_1 u_1(t) + \mu) I(t) \right)\\
&\, \, + \lambda_4(t) \left((1 - \phi) \delta L_1(t)
- \sigma \frac{\beta}{N} I(t) L_2(t)
- (\omega + \epsilon_2 u_2(t) + \tau_2 + \mu)L_2(t) \right)\\
&\, \,  + \lambda_5(t) \left((\tau_0 + \epsilon_1 u_1(t) )I(t)
+ \tau_1 L_1(t) + (\tau_2 +\epsilon_2 u_2(t)) L_2(t)
- \sigma_R \frac{\beta}{N} I(t) R(t) - (\omega_R + \mu)R(t) \right)
\end{split}
\end{equation*}
is called the \emph{Hamiltonian}, and the minimization condition
\begin{equation}
\label{maxcondPMP}
\begin{split}
H(S^*(t), &L_1^*(t), I^*(t), L_2^*(t), R^*(t),
\lambda^*(t), u_1^*(t), u_2^*(t))\\
&= \min_{0 \leq u_1, u_2 \leq 1}
H(S^*(t), L_1^*(t), I^*(t), L_2^*(t), R^*(t), \lambda^*(t), u_1, u_2)
\end{split}
\end{equation}
holds almost everywhere on $[0, T]$. Moreover, the transversality conditions
\begin{equation*}
\lambda_i(T) = 0, \quad
i =1,\ldots, 5 \, ,
\end{equation*}
hold.

\begin{theorem}
\label{the:thm}
Problem \eqref{modelGab_controls}, \eqref{mincostfunct} with fixed initial conditions
$S(0)$, $L_1(0)$, $I(0)$, $L_2(0)$ and $R(0)$ and fixed final time $T$, admits an unique
optimal solution $\left(S^*(\cdot), L_1^*(\cdot), I^*(\cdot), L_2^*(\cdot), R^*(\cdot)\right)$
associated to an optimal control pair $\left(u_1^*(\cdot), u_2^*(\cdot)\right)$ on $[0, T]$.
Moreover, there exists adjoint functions $\lambda_1^*(\cdot)$, $\lambda_2^*(\cdot)$,
$\lambda_3^*(\cdot)$, $\lambda_4^*(\cdot)$ and $\lambda_5^*(\cdot)$ such that
\begin{equation}
\label{adjoint_function}
\begin{cases}
\dot{\lambda^*_1}(t) = \lambda^*_1(t) \left(\frac{\beta}{N} I^*(t)
+ \mu \right) - \lambda^*_2(t) \frac{\beta}{N} I^*(t) \\[0.1 cm]
\dot{\lambda^*_2}(t) = \lambda^*_2(t)\left(\delta + \tau_1 + \mu\right)
- \lambda^*_3(t) \phi \delta - \lambda^*_4(t) (1 - \phi) \delta
- \lambda^*_5(t)\tau_1 \\[0.1 cm]
\dot{\lambda^*_3}(t) = -1 + \lambda^*_1(t) \frac{\beta}{N} S^*(t)
- \lambda^*_2(t) \frac{\beta}{N}(S^*(t) + \sigma L_2^*(t) + \sigma_R R^*(t))\\
\qquad \quad + \lambda^*_3(t) \left(\tau_0 +\epsilon_1 u_1^*(t) + \mu\right)
+\lambda^*_4(t)\sigma \frac{\beta}{N} L_2^*(t)
- \lambda^*_5(t)\left(\tau_0 + \epsilon_1 u^*_1(t)
- \sigma_R \frac{\beta}{N} R^*(t) \right) \\[0.1 cm]
\dot{\lambda^*_4}(t) = - 1 - \lambda^*_2(t) \frac{\beta}{N}I^*(t)
\sigma - \lambda^*_3(t) \omega + \lambda^*_4(t)\left(\sigma \frac{\beta}{N} I^*(t)
+ \omega + \epsilon_2 u^*_2(t) + \tau_2 + \mu\right)\\
\qquad \quad - \lambda^*_5(t)\left( \tau_2 + \epsilon_2 u^*_2(t) \right) \\[0.1 cm]
\dot{\lambda^*_5}(t) = -\lambda^*_2(t) \sigma_R \frac{\beta}{N}I^*(t)
- \lambda^*_3(t) \omega_R + \lambda^*_5(t)\left(\sigma_R \frac{\beta}{N} I^*(t)
+\omega_R + \mu\right) \, ,
\end{cases}
\end{equation}
with transversality conditions
\begin{equation*}
\lambda^*_i(T) = 0,
\quad i=1, \ldots, 5 \, .
\end{equation*}
Furthermore,
\begin{equation}
\label{optcontrols}
\begin{split}
u_1^*(t) &= \min \left\{ \max \left\{0, \frac{\epsilon_1 I^*
\left(\lambda^*_3 - \lambda^*_5\right)}{W_1}\right\}, 1 \right\} \, ,\\
u_2^*(t) &= \min \left\{ \max \left\{0, \frac{\epsilon_2 L^*_2
\left(\lambda^*_4 - \lambda^*_5\right)}{W_2}\right\}, 1 \right\}  \, .
\end{split}
\end{equation}
\end{theorem}

\begin{proof}
Existence of an optimal solution $\left(S^*, L_1^*, I^*, L_2^*, R^*\right)$
associated to an optimal control pair $\left(u_1^*, u_2^*\right)$ comes from
the convexity of the integrand of the cost function $J$ with respect
to the controls $(u_1, u_2)$ and the Lipschitz property of the state system
with respect to state variables $\left(S, L_1, I, L_2, R\right)$
(see, \textrm{e.g.}, \cite{Cesari_1983,Fleming_Rishel_1975}).
System \eqref{adjoint_function} is derived from the Pontryagin maximum principle
(see \eqref{adjsystemPMP}, \cite{Pontryagin_et_all_1962})
and the optimal controls \eqref{optcontrols} come from the minimization condition \eqref{maxcondPMP}.
For small final time $T$, the optimal control pair given by \eqref{optcontrols}
is unique due to the boundedness of the state and adjoint functions and the Lipschitz property
of systems \eqref{modelGab_controls} and \eqref{adjoint_function}
(see \cite{SLenhart_2002} and references cited therein).
\end{proof}

\begin{remark}
Due to the fact that the state system \eqref{modelGab_controls} is autonomous,
the proof of Theorem~\ref{the:thm} is valid for any time $T$
and not only for small time $T$.
\end{remark}


\section{Numerical results and discussion}
\label{sec:num:results}

In this section we present results of the numerical implementation
of optimal control strategies for the TB model \eqref{modelGab_controls}
with cost functional \eqref{costfunction}.
We consider variations of some parameters separately,
and interpret the obtained optimal controls
and associated optimal state solutions.

Different approaches were used to obtain and confirm the numerical results.
One approach consisted in using IPOPT \cite{IPOPT}
and the algebraic modeling language AMPL \cite{AMPL}.
A second approach was to use the PROPT Matlab Optimal Control Software \cite{PROPT}.
The results coincide with the ones obtained by an iterative method
that consists in solving the system of ten ODEs given by
\eqref{modelGab_controls} and \eqref{adjoint_function}.
For that, first we solve system \eqref{modelGab_controls}
with a guess for the controls over the time interval
$[0, T]$ using a forward fourth-order Runge--Kutta scheme
and the transversality conditions $\lambda_i(T) = 0$, $i=1, \ldots, 5$.
Then, system \eqref{adjoint_function} is solved by
a backward fourth-order Runge--Kutta scheme using the current
iteration solution of \eqref{modelGab_controls}.
The controls are updated by using a convex combination of the previous controls
and the values from \eqref{optcontrols}. The iteration is stopped when the values of the unknowns
at the previous iteration are very close to the ones at the present iteration.
For more details see, \emph{e.g.}, \cite{SLenhart_2002}.

\begin{table}[!htb]
\centering
\begin{tabular}{|l | l | l |}
\hline
{\scriptsize{Symbol}} & {\scriptsize{Description}}  & {\scriptsize{Value}} \\
\hline
{\scriptsize{$\beta$}} & {\scriptsize{Transmission coefficient}}
& {\scriptsize{$75, 100, 150, 175$ }}\\
{\scriptsize{$\mu$}} & {\scriptsize{Death and birth rate}}
& {\scriptsize{ $1/70 \, yr^{-1}$}}\\
{\scriptsize{$\delta$}} & {\scriptsize{Rate at which individuals leave $L_1$}}
& {\scriptsize{$12 \, yr^{-1}$}}\\
{\scriptsize{$\phi$}} & {\scriptsize{Proportion of individuals going to $I$}}
& {\scriptsize{$0.05$}}\\
{\scriptsize{$\omega$}} & {\scriptsize{Rate of endogenous reactivation for persistent latent infections}}
& {\scriptsize{$0.0002 \, yr^{-1}$}}\\
{\scriptsize{$\omega_R$}} & {\scriptsize{Rate of endogenous reactivation for treated individuals}}
&{\scriptsize{ $0.00002 \, yr^{-1}$}}\\
{\scriptsize{$\sigma$}} & {\scriptsize{Factor reducing the risk of infection as a result of acquired}}  & \\
& {\scriptsize{immunity to a previous infection for $L_2$}} & {\scriptsize{$0.25$}} \\
{\scriptsize{$\sigma_R$}} & {\scriptsize{Rate of exogenous reinfection of treated patients}}
& {\scriptsize{0.25}} \\
{\scriptsize{$\tau_0$}} & {\scriptsize{Rate of recovery under treatment of active TB}}
&  {\scriptsize{$2 \, yr^{-1}$}}\\
{\scriptsize{$\tau_1$}} & {\scriptsize{Rate of recovery under treatment of latent individuals $L1$}}
&  {\scriptsize{$2 \, yr^{-1}$}}\\
{\scriptsize{$\tau_2$}} & {\scriptsize{Rate of recovery under treatment of latent individuals $L2$}}
&  {\scriptsize{$1 \, yr^{-1}$}}\\
{\scriptsize{$N$}} & {\scriptsize{Total population}} & {\scriptsize{$10000, 15000, 30000$}} \\
{\scriptsize{$S(0)$}} & {\scriptsize{Initial number of susceptible individuals}}
& {\scriptsize{$\frac{76}{120}N$}}  \\
{\scriptsize{$L_1(0)$}} & {\scriptsize{Initial number of early latent $L_1$ individuals}}
& {\scriptsize{$\frac{37}{120}N$}}\\
{\scriptsize{$I(0)$}} & {\scriptsize{Initial number of infectious individuals}}
&  {\scriptsize{$\frac{4}{120}N$}}\\
{\scriptsize{$L_2(0)$}} & {\scriptsize{Initial number of persistent latent $L_2$ individuals}}
& {\scriptsize{$\frac{2}{120}N$}} \\
{\scriptsize{$R(0)$}} & {\scriptsize{Initial number of recovered individuals}}
& {\scriptsize{$\frac{1}{120}N$}} \\
{\scriptsize{$T$}} & {\scriptsize{Total simulation duration}} & {\scriptsize{5 yr}} \\
{\scriptsize{$\epsilon_1$}} & {\scriptsize{Efficacy of treatment of active TB $I$}}
& {\scriptsize{$0.25, 0.5, 0.75$}} \\
{\scriptsize{$\epsilon_2$}} & {\scriptsize{Efficacy of treatment of latent TB $L_2$}}
& {\scriptsize{$0.25, 0.5, 0.75$}} \\
{\scriptsize{$W_1$}} & {\scriptsize{Weight constant on control $u_1(t)$}}
& {\scriptsize{$150, 250, 500$}}\\
{\scriptsize{$W_2$}} & {\scriptsize{Weight constant on control $u_2(t)$}}
& {\scriptsize{$50, 150, 250$}}\\
\hline
\end{tabular}
\caption{Parameter values.}
\label{parameters}
\end{table}

We start comparing the case of minimizing the number of infectious
and persistent latent individuals, $I + L_2$, with and without controls.
We consider $\beta = 100$, $N=30000$, $\epsilon_1 = \epsilon_2 = 0.5$,
$W_1 = 500$, $W_2 = 50$, and the values of the remaining parameters
are presented in Table~\ref{parameters}.
For these parameter values, $R_0(0, 0) = 2.2$ and $R_0(1, 1) = 1.76$.
In Figure~\ref{fig:com:e:sem:controlos} we observe that the fraction
of active infectious and persistent latent individuals
is lower when controls are considered. More precisely, at the end of five years,
the total number of infectious and persistent latent individuals $I + L_2$
is 320 when controls are considered, and 1550 without controls.
To minimize the total number of infectious and persistent latent individuals,
the optimal control $u_1$ is at the upper bound during 2.3 years and
then, during the remaining 2.7 years, it decreases to the lower bound.
The control $u_2$ is at the upper bound during almost 4.7 years
(see Figure~\ref{fig:u1u2:com:e:sem:controlos}).

\begin{figure}[!htb]
\centering
\includegraphics[width=0.5\textwidth]{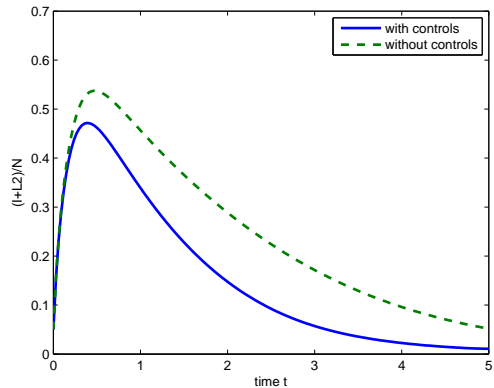}
\caption{$(I + L_2)/N$ with and without controls
for $\beta = 100$, $N=30000$, $\epsilon_1 = \epsilon_2 = 0.5$, $W_1 = 500$ and $W_2 = 50$.}
\label{fig:com:e:sem:controlos}
\end{figure}

\begin{figure}[!htb]
\centering
\includegraphics[width=0.5\textwidth]{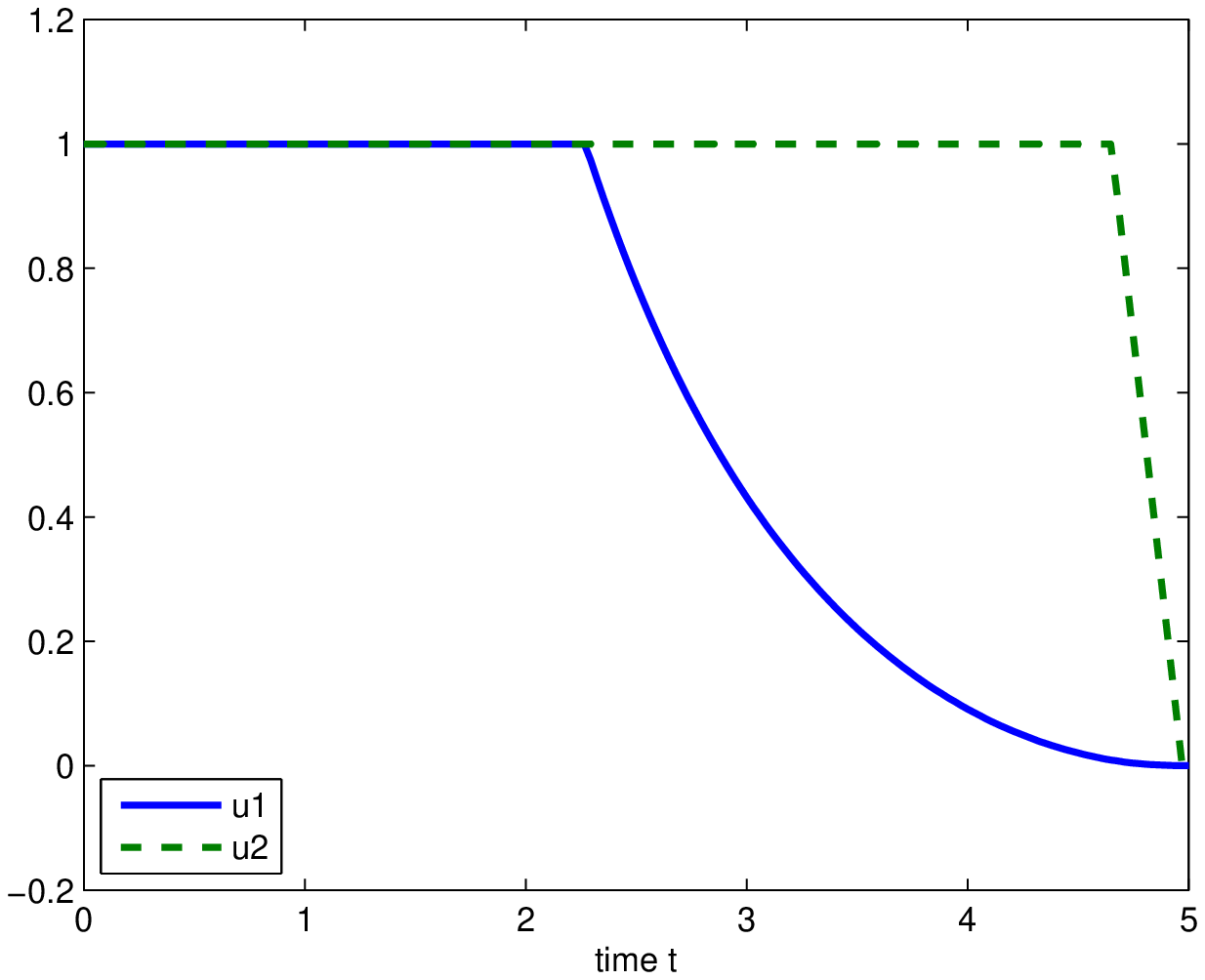}
\caption{Controls $u_1$ and $u_2$ for $\beta =100$, $N=30000$,
$\epsilon_1 = \epsilon_2 = 0.5$, $W_1 = 500$, $W_2 = 50$.}
\label{fig:u1u2:com:e:sem:controlos}
\end{figure}

We test the relevance of the optimal control strategies, given by controls
$u_1$ and $u_2$, in the reduction of the fraction of active infected individuals
$I/N$ and persistent latent individuals $L_2/N$. We observe in
Figure~\ref{fig:I:L2:com:e:sem:controlos} that both fractions $I/N$ and $L_2/N$
are lower when controls are considered, that is, we can conclude that the
implementation of the measures to prevent the failure of treatment
in active TB infectious individuals $I$ and the increase of number
of persistent latent individuals $L_2$, that are identified
and put under treatment, are good policies for the achievement
of our goal. Some of the policies associated to the control $u_1$ are
the supervision and the support of active TB infectious individuals $I$.
It is important to ensure that active TB infectious individuals $I$ complete the treatment,
which is difficult due to its duration and second effects.
The supervision can be made, however, not only in hospitals but also paying to specialized people
to go to patients home. This implies higher monetary cost, that is,
greater values for $W_i$, which is illustrated
in Figures~\ref{fig:variar:W2}--\ref{fig:W1igualW2}.

\begin{figure}[!htb]
\centering
\subfloat[\footnotesize{$I/N$ with and without controls}]{\label{fig:I:com:e:sem:controlos}
\includegraphics[width=0.45\textwidth]{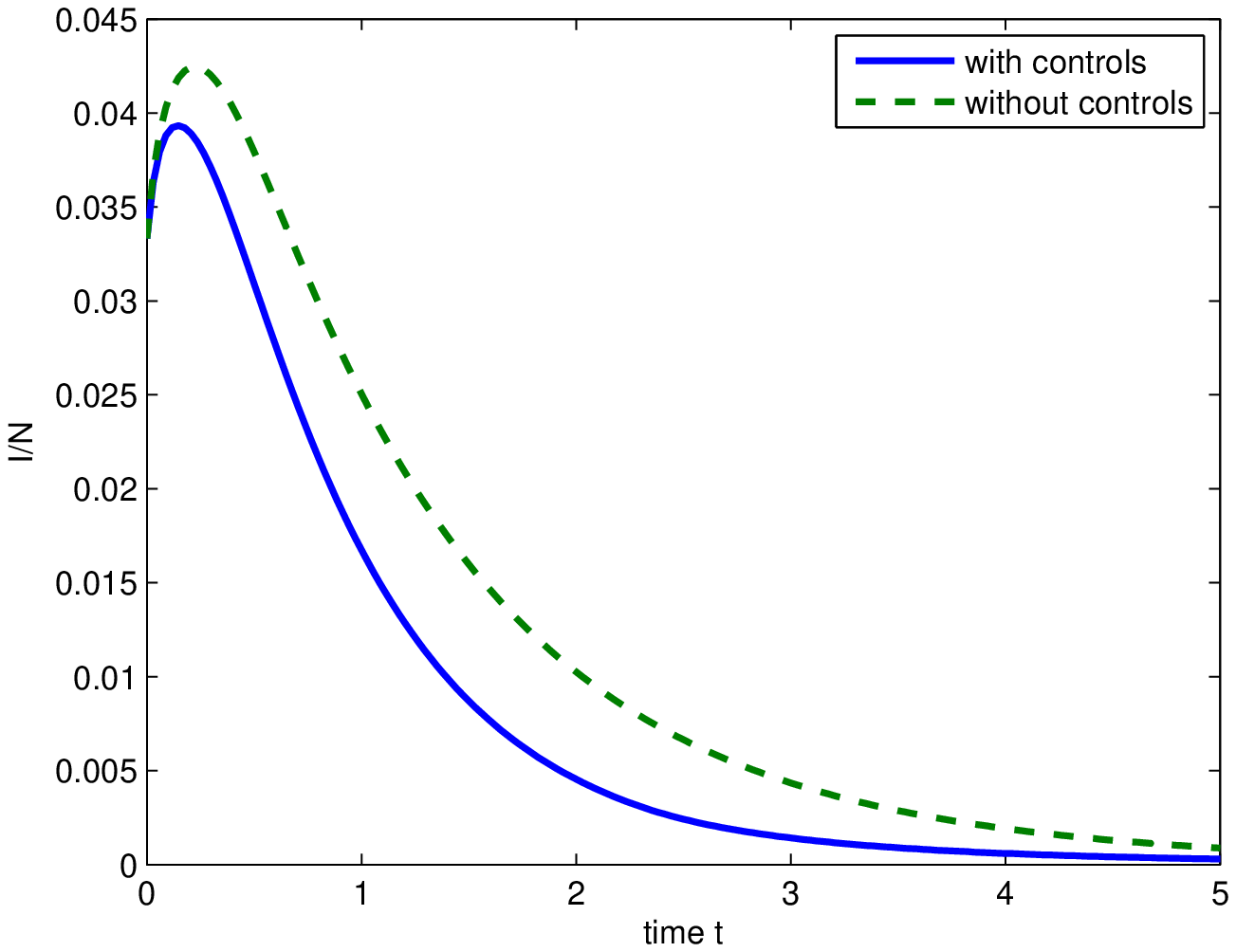}}
\subfloat[\footnotesize{$L_2/N$ with and without controls}]{\label{fig:L2:com:e:sem:controlos}
\includegraphics[width=0.45\textwidth]{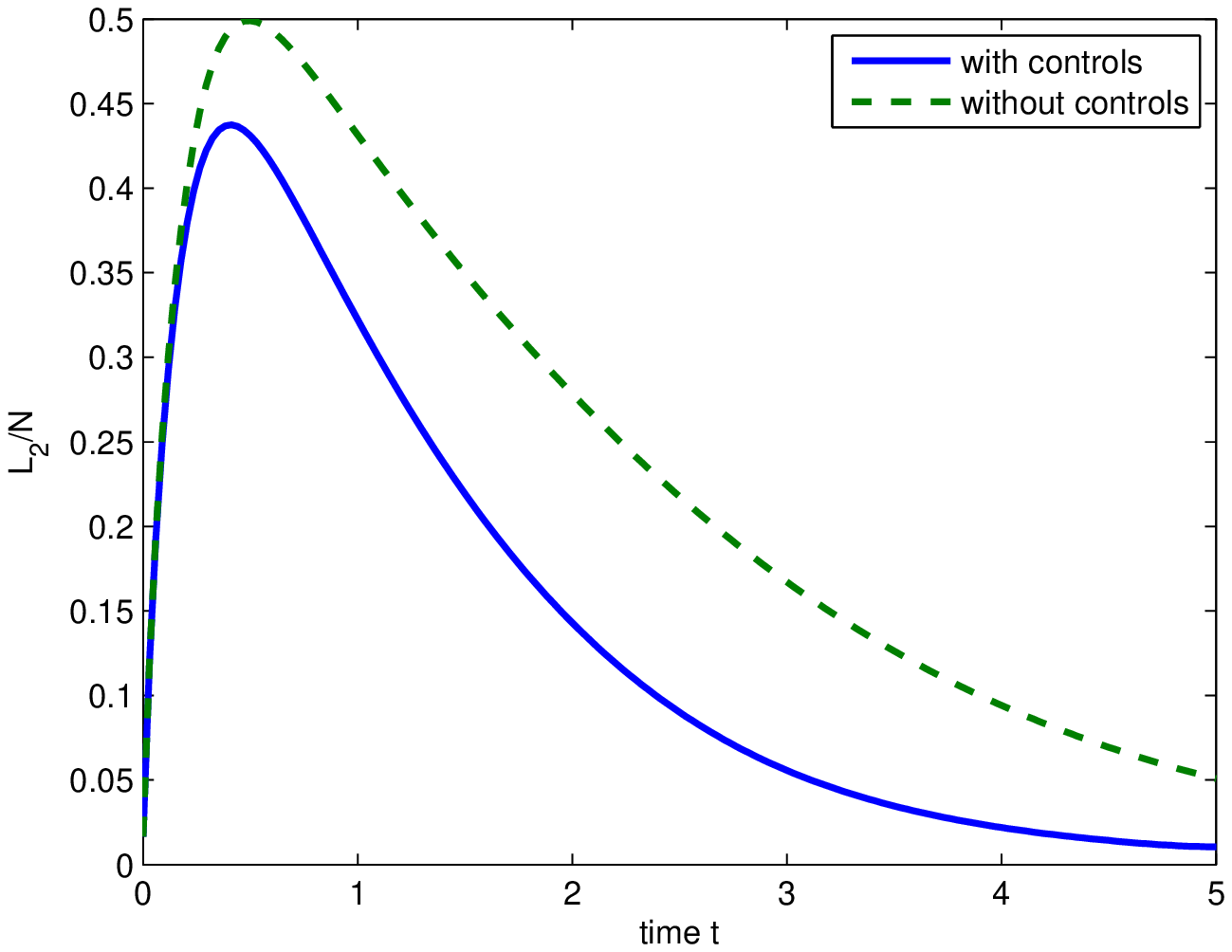}}
\caption{$I/N$ and $L_2/N$, with and without controls,
for $\beta = 100$, $N=30000$, $\epsilon_1 = \epsilon_2 = 0.5$, $W_1 = 500$ and $W_2 = 50$.}
\label{fig:I:L2:com:e:sem:controlos}
\end{figure}

In what follows, we consider always strategies using controls
$u_1(t)$ and $u_2(t)$ with $t \in [0, 5]$.
Figure~\ref{fig:variar:beta} illustrates how optimal control
strategies change as the transmission coefficient parameter $\beta$ varies.
We consider four different values for $\beta$, $75$, $100$, $150$ and $175$,
the other parameters taking the values $N=30000$, $\epsilon_1 = \epsilon_2 = 0.5$,
$W_1 = 500$, and $W_2 = 50$ (the values of the remaining parameters
are presented in Table~\ref{parameters}). All the values that $\beta$
takes correspond to the case where the disease may become endemic, \textrm{i.e.},
$R_0 > 1$. We observe that as the parameter $\beta$ increases, the control $u_1$
is at the upper bound for a longer period of time, but the variation on the control
$u_2$ is not so significant. In Figure~\ref{fig:variar:beta} (c) one can see that
as $\beta$ decreases the fraction of infectious and persistent
latent individuals also decreases, as expected.

\begin{figure}[!htb]
\centering
\subfloat[\footnotesize{Control $u_1$}]{\label{fig:variarbeta:u1}
\includegraphics[width=0.32\textwidth]{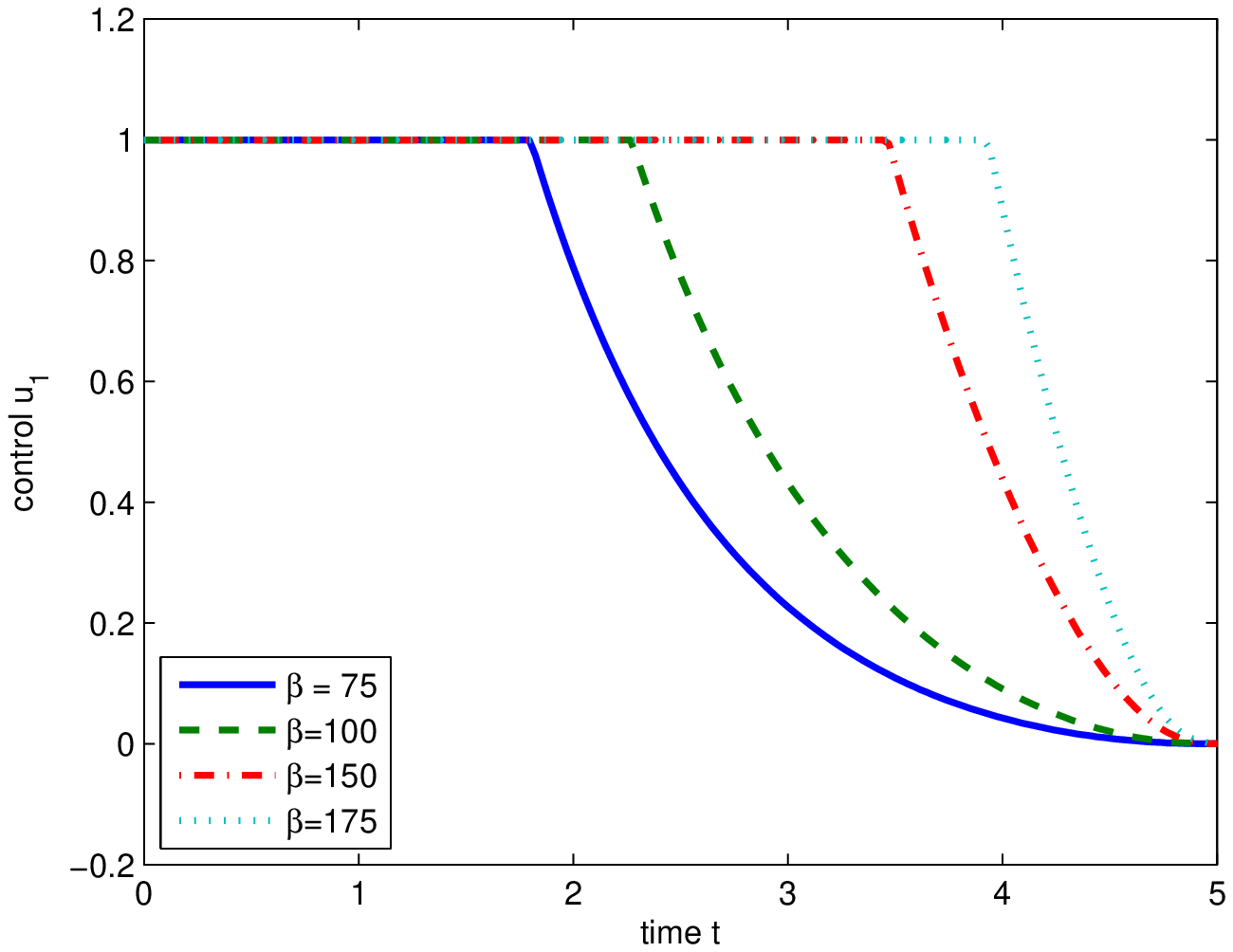}}
\subfloat[\footnotesize{Control $u_2$}]{\label{fig:variarbeta:u2}
\includegraphics[width=0.32\textwidth]{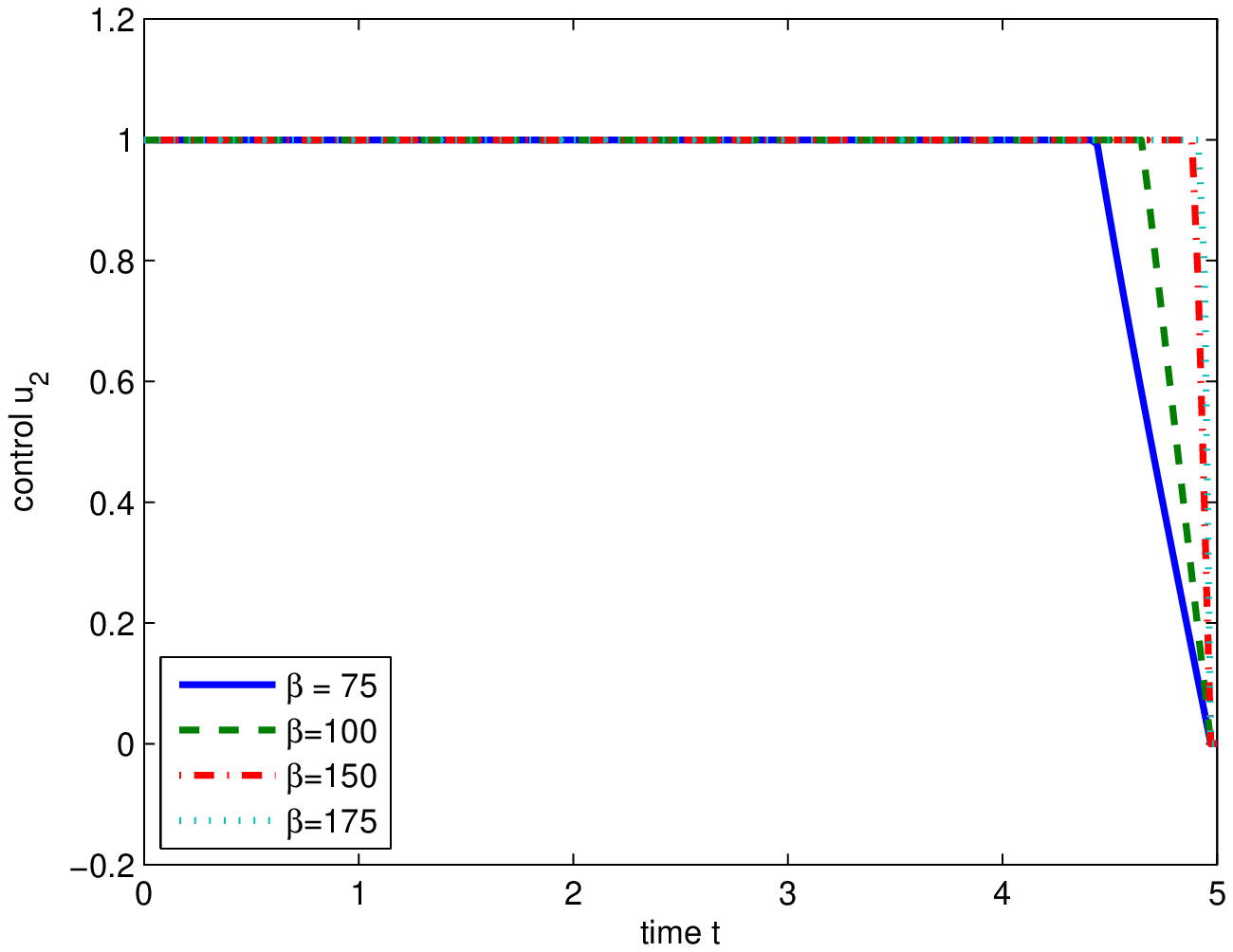}}
\subfloat[\footnotesize{$(I+L_2)/30000$}]{\label{fig:variarbeta:I:L2}
\includegraphics[width=0.32\textwidth]{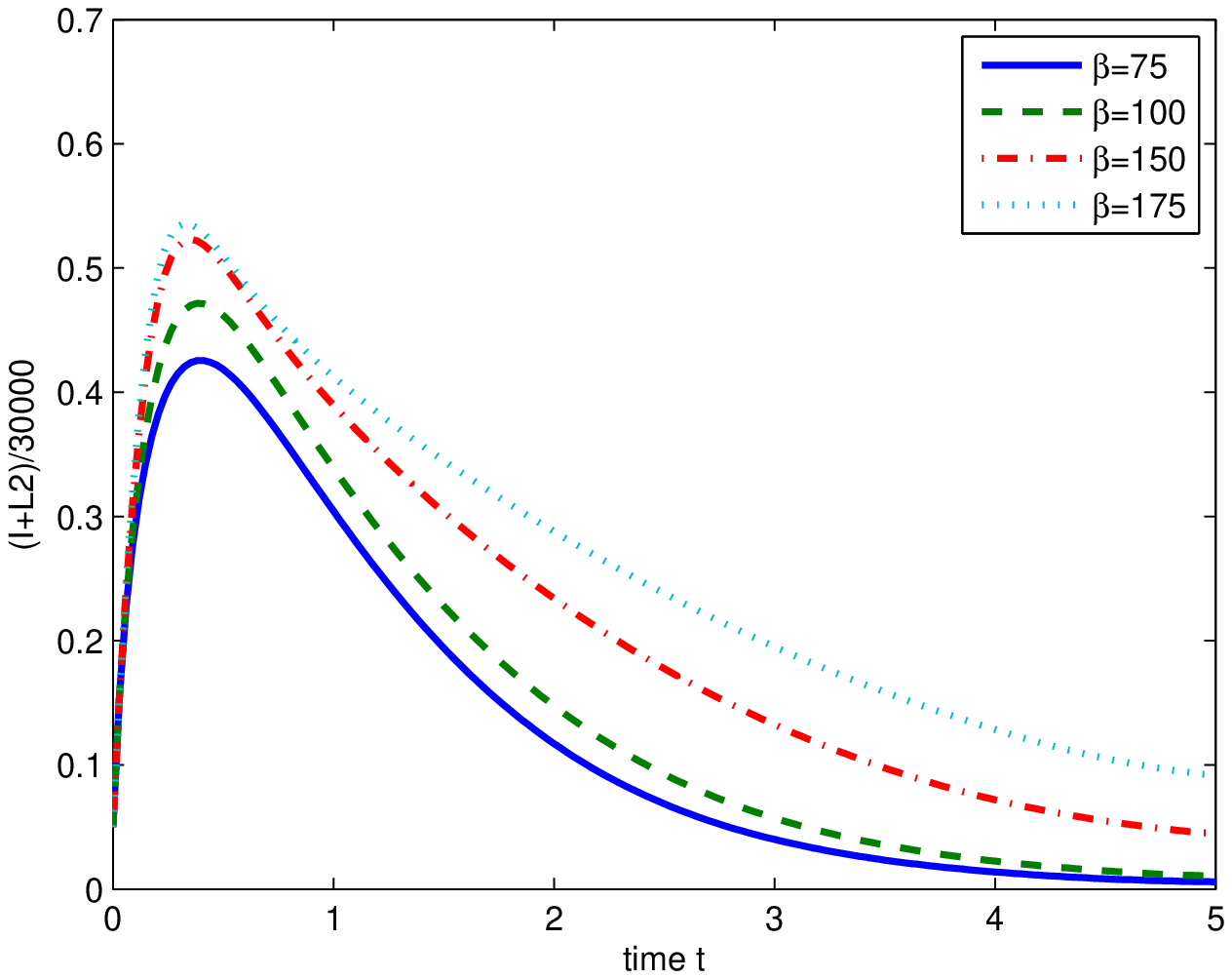}}
\caption{Variation of parameter $\beta$ ($\beta = 75, 100, 150, 175$).}
\label{fig:variar:beta}
\end{figure}

In Figure~\ref{fig:variar:N} the size of the population takes three different values,
$30000$, $40000$ and $60000$, and the controls $u_1(t)$ and $u_2(t)$ are plotted
for $t \in [0, 5]$. One can conclude that the optimal strategies do not vary significantly:
the control $u_1$ is at the upper bound during 2.5 years for $N=40000$
and during 2.8 years for $N=60000$, and $u_2$ is at the upper bound during 4.7 years
for $N=40000$ and during 4.8 years for $N=60000$. On the other hand,
the controls $\left(u_1(\cdot), u_2(\cdot)\right)$ are such that the fraction
of infectious and persistent latent individuals $\left(I + L_2\right)/N$
does not depend on the size of the population.

\begin{figure}[!htb]
\centering
\subfloat[\footnotesize{Control $u_1$}]{\label{fig:variarN:u1}
\includegraphics[width=0.32\textwidth]{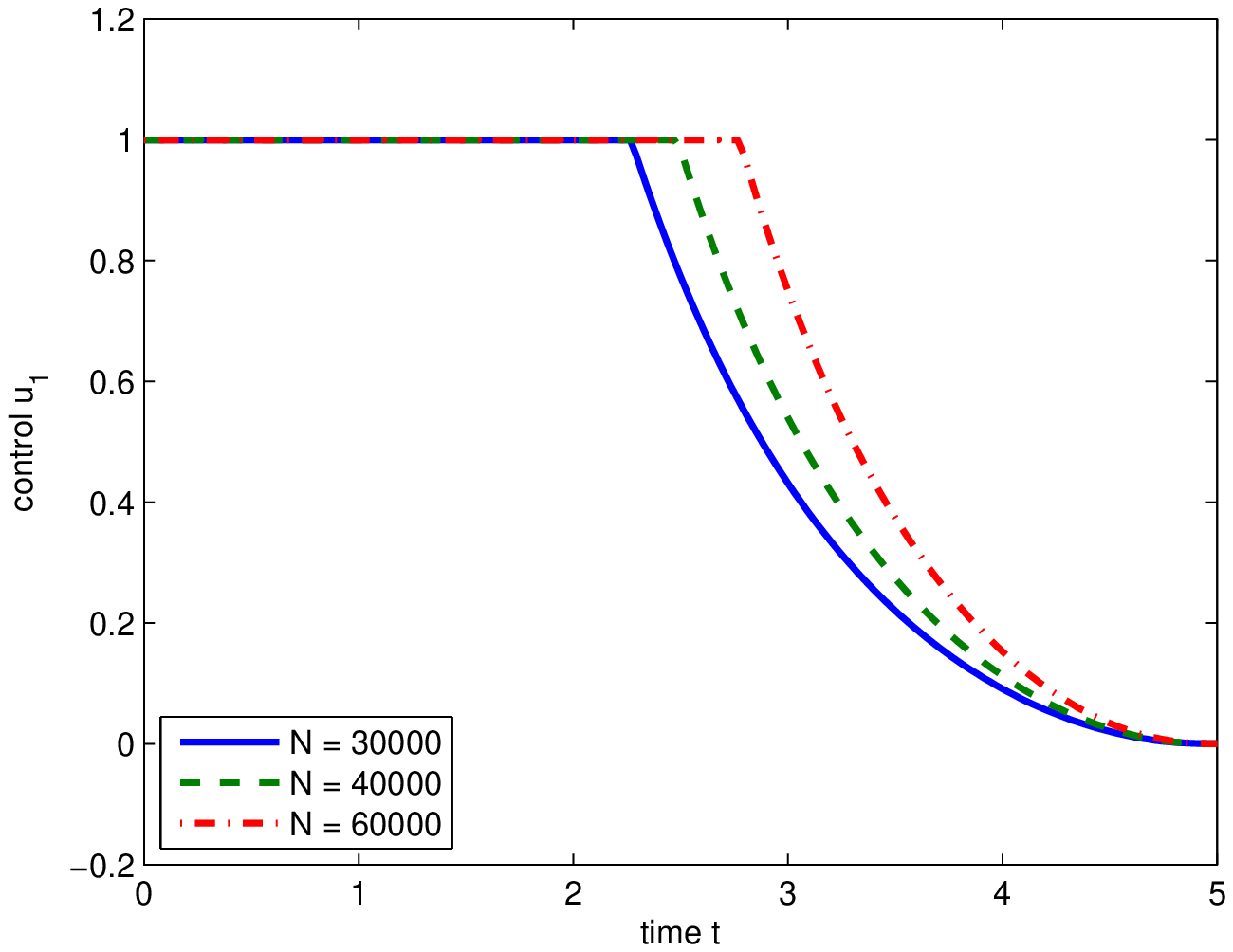}}
\subfloat[\footnotesize{Control $u_2$}]{\label{fig:variarN:u2}
\includegraphics[width=0.32\textwidth]{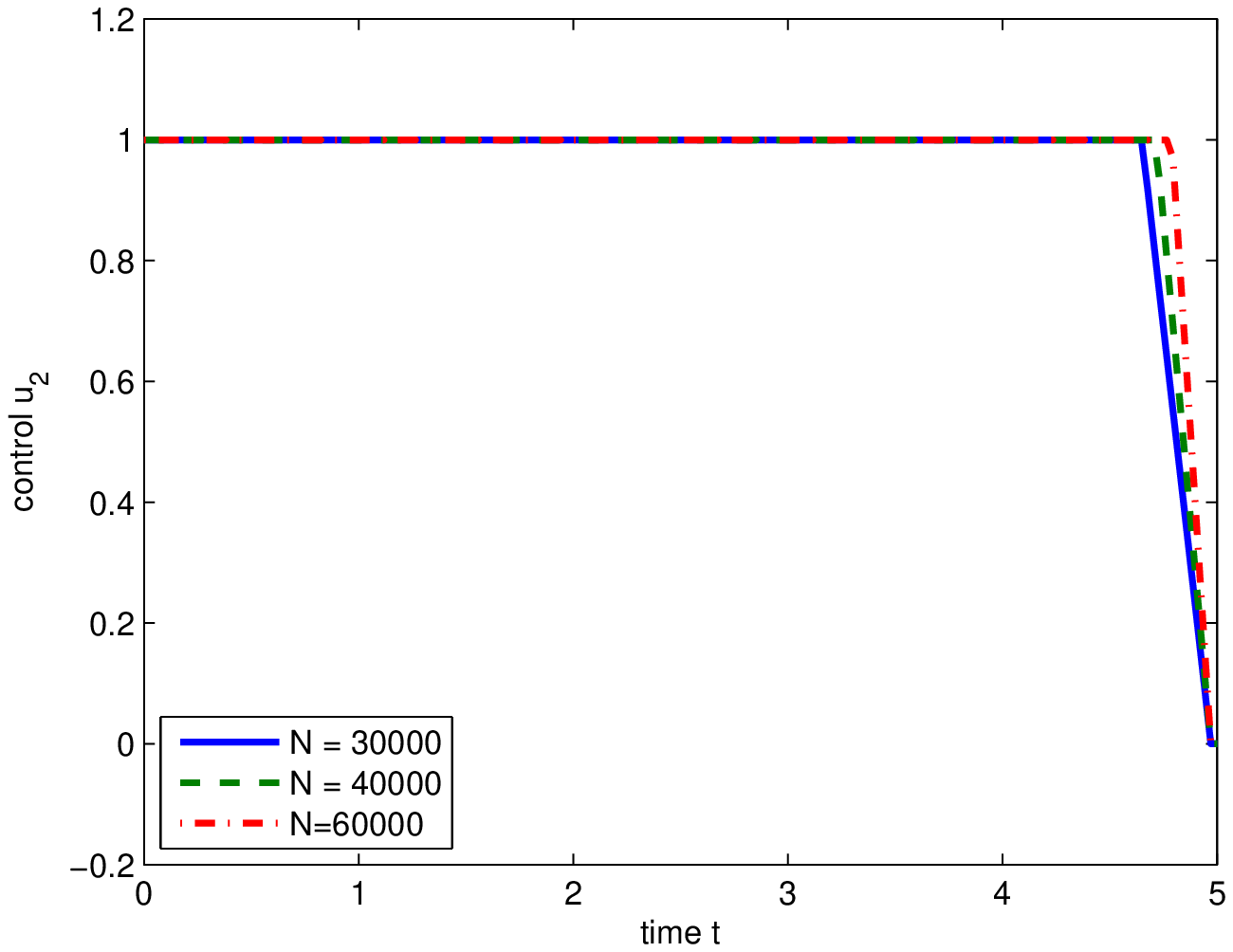}}
\subfloat[\footnotesize{$(I+L_2)/N$}]{\label{fig:variarN:I:L2}
\includegraphics[width=0.32\textwidth]{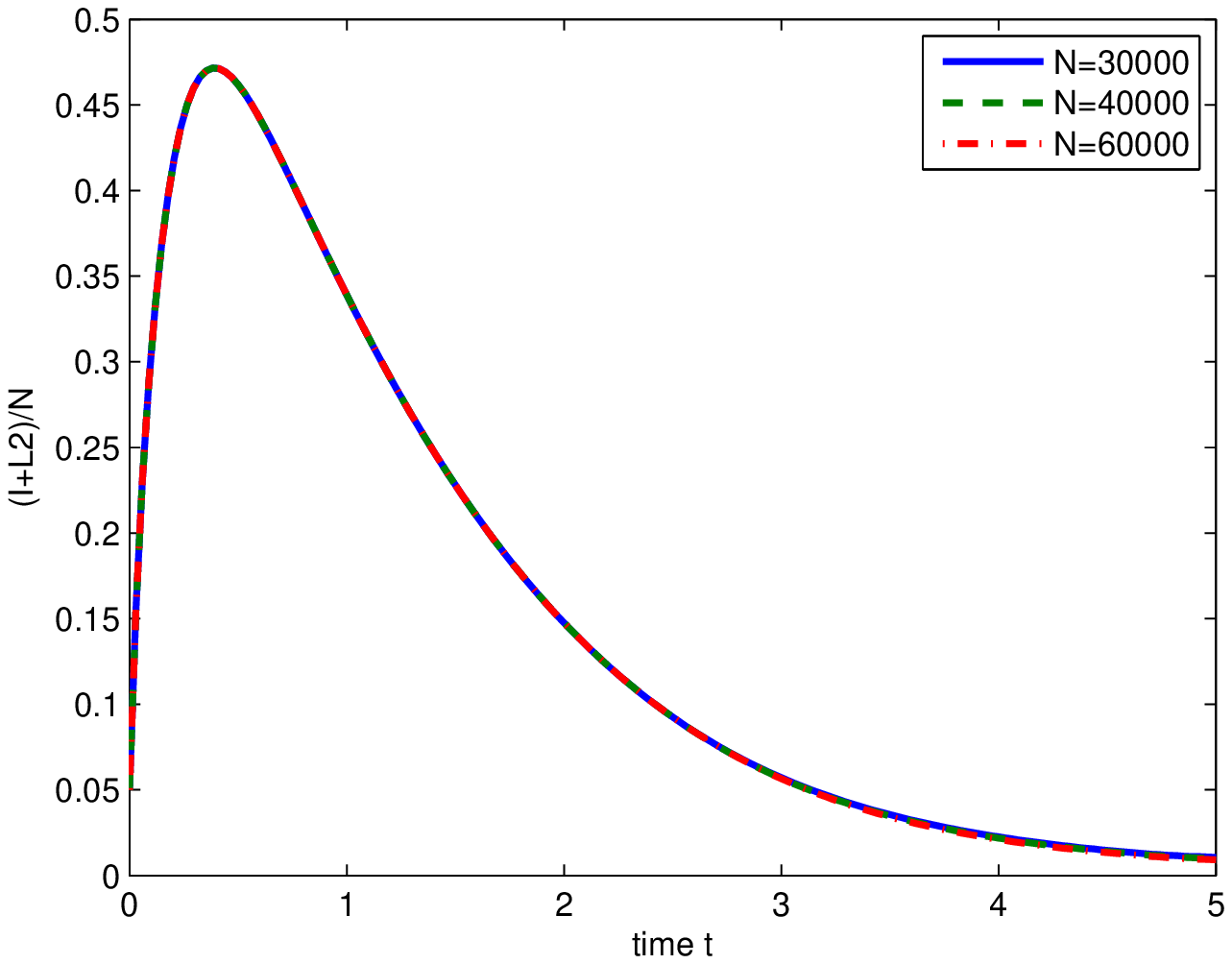}}
\caption{Different sizes $N$ of population ($N = 30000$, $40000$, $60000$)
with $\beta = 100$, $W_1 = 500$, $W_2 = 50$, $\epsilon_1 = \epsilon_2= 0.5$,
and the values of the parameters of Table~\ref{parameters}.}
\label{fig:variar:N}
\end{figure}

In Figures~\ref{fig:variar:W2}--\ref{fig:W1igualW2} we present
the effect of the weight constants $W_i$, $i=1, 2$, on the controls,
where $\beta = 100$, $N= 30000$, and $\epsilon_1 = \epsilon_2 = 0.5$
(and the other parameters are given by Table~\ref{parameters}).
We assume that the weight constant $W_1$ associated with the control $u_1$
is greater or equal than the weight constant $W_2$ associated
with the control $u_2$, because the cost associated to $u_1$ includes
the cost of holding active infected patients $I$ in the hospital
or paying people to supervise the patients, assuring that they finish
their treatment, and the cost associated to $u_2$ is related to the fraction
of persistent latent individuals $L_2$ that is put under treatment.
It is clear that when $W_1$ is fixed and $W_2$ increases,
the amount of $u_2$ decreases and $u_1$ remains the same
(Figure~\ref{fig:variar:W2}). Analogously, when $W_2$ is fixed and $W_1$ decreases,
the amount of $u_1$ increases (Figure~\ref{fig:variar:W1}). When the weight constants
are equal (Figure~\ref{fig:W1igualW2}), both controls vary but, in all the three cases,
the optimal control strategies assure the same value for the fraction
of infectious and persistent latent individuals.
Figure~\ref{fig:variar:epsilon} illustrates the situation when we vary
the measures of control efficacy $\epsilon_i$, $i=1, 2$. We consider $\beta =100$,
$N=30000$, $W_1 = 500$, $W_2 = 50$, and $\epsilon_1 = \epsilon_2$ with
$\epsilon_i = 0.25; 0.5; 0.75$, $i=1, 2$. We observe that as the efficacy
of the controls increase, the control strategies contribute to the minimization
of the fraction of infectious and persistent latent individuals.

\begin{figure}[!htb]
\centering
\subfloat[\footnotesize{Control $u_1$}]{\label{fig:variarW2:u1}
\includegraphics[width=0.32\textwidth]{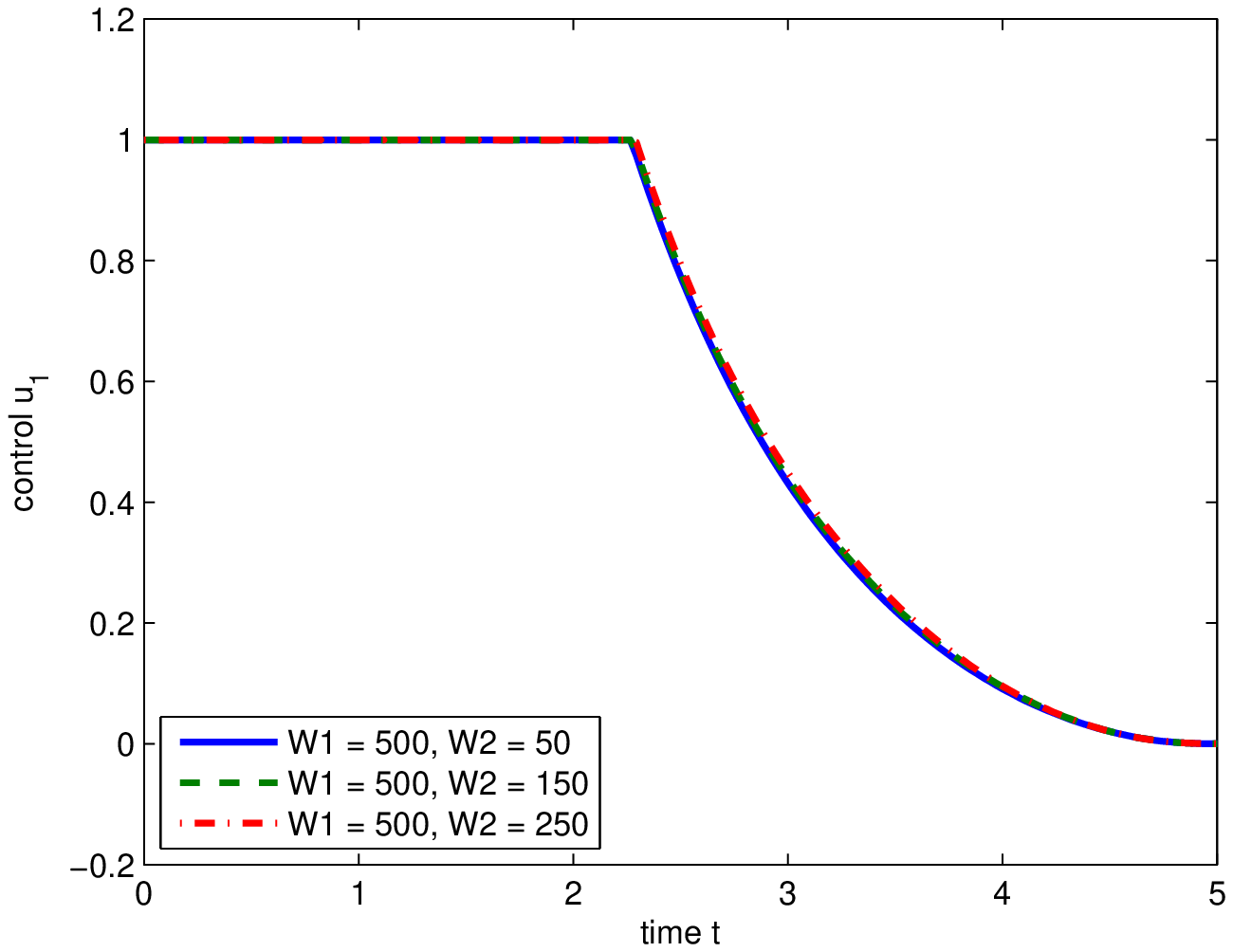}}
\subfloat[\footnotesize{Control $u_2$}]{\label{fig:variarW2:u2}
\includegraphics[width=0.32\textwidth]{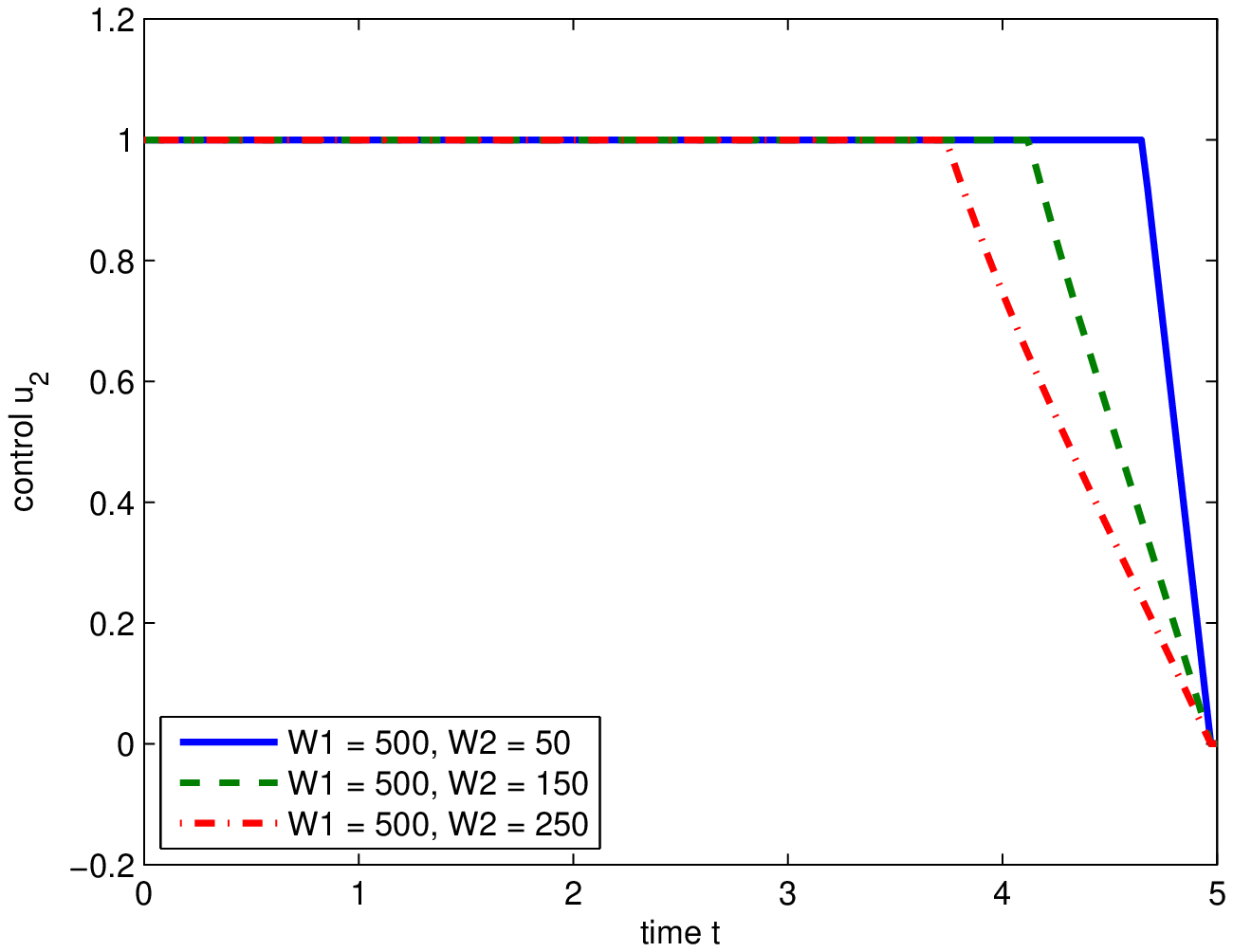}}
\subfloat[\footnotesize{$(I+L_2)/N$}]{\label{fig:variarW2:I:L2}
\includegraphics[width=0.32\textwidth]{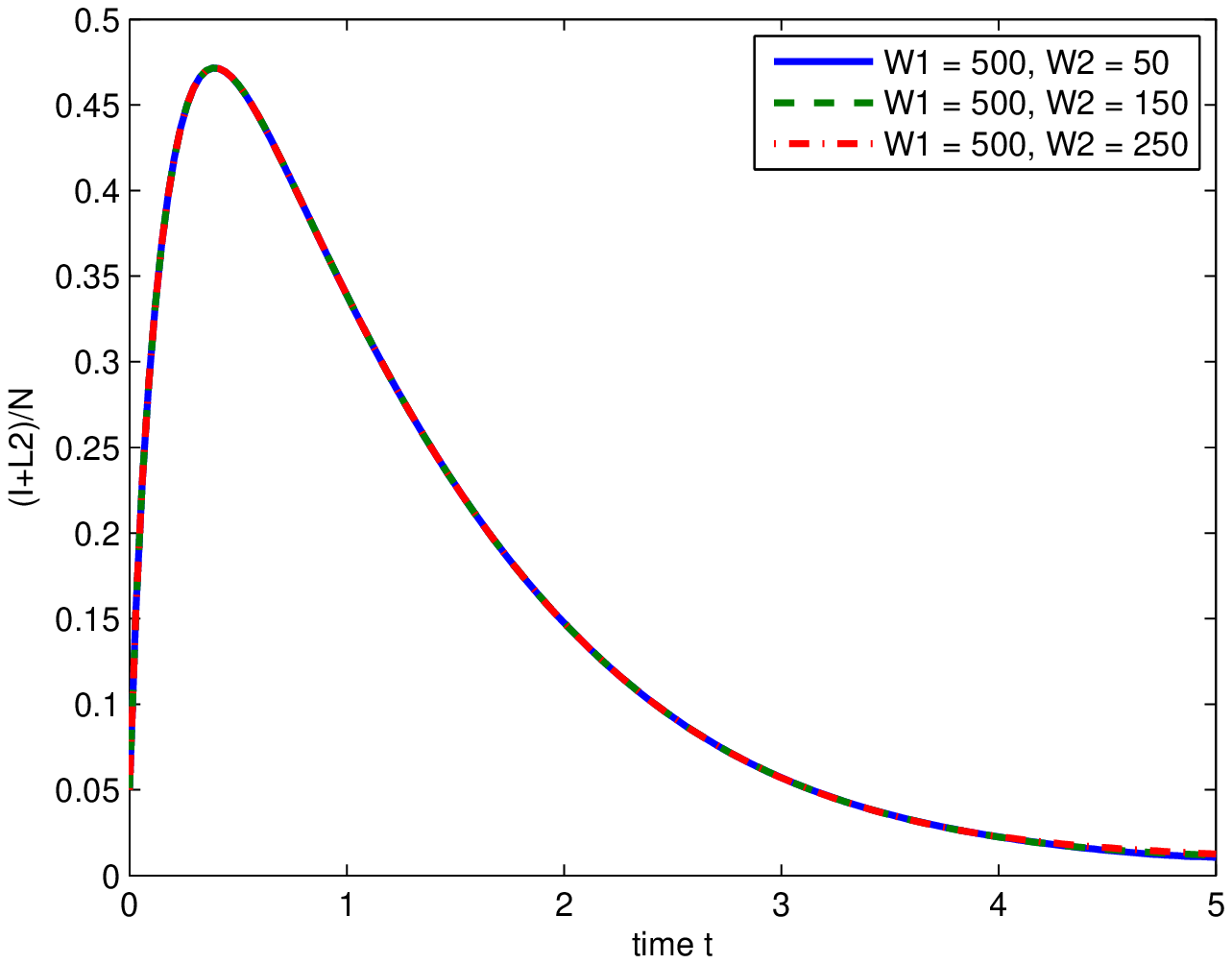}}
\caption{Variation on the control weights: fix $W_1 = 500$ and vary $W_2$ ($W_2 = 50, 150, 250$).}
\label{fig:variar:W2}
\end{figure}

\begin{figure}[!htb]
\centering
\subfloat[\footnotesize{Control $u_1$}]{\label{fig:variarW1:u1}
\includegraphics[width=0.32\textwidth]{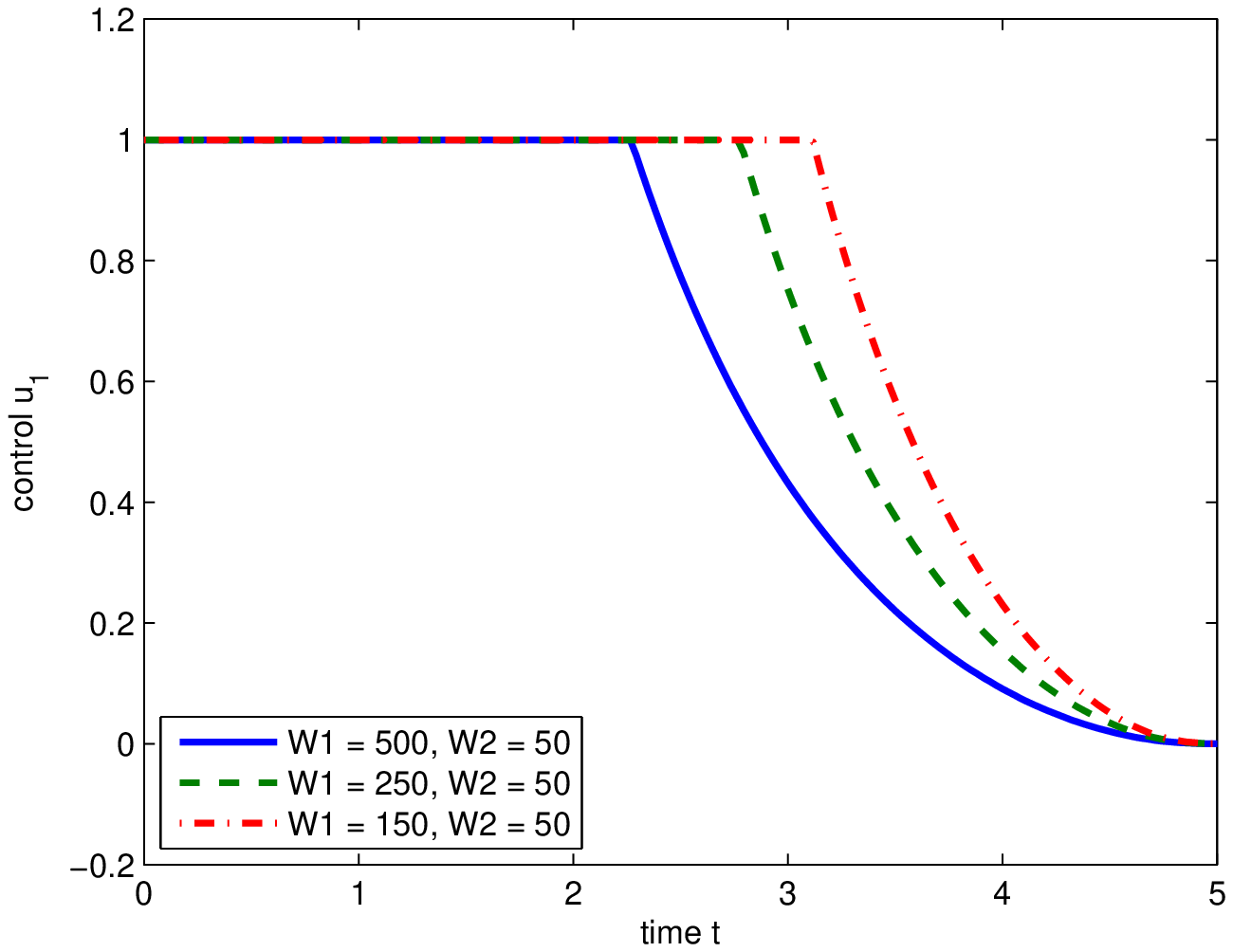}}
\subfloat[\footnotesize{Control $u_2$}]{\label{fig:variarW1:u2}
\includegraphics[width=0.32\textwidth]{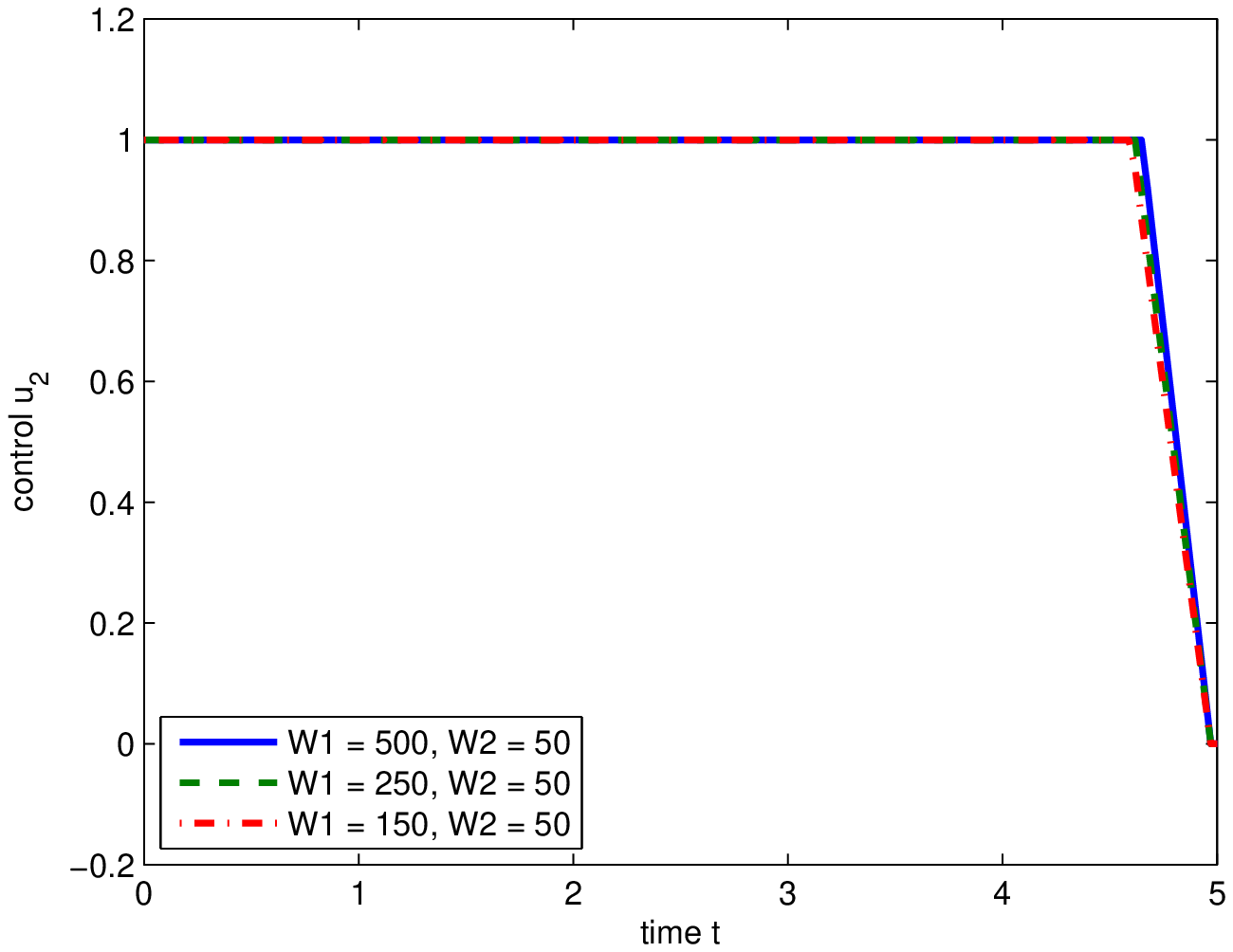}}
\subfloat[\footnotesize{$(I+L_2)/N$}]{\label{fig:variarW1:I:L2}
\includegraphics[width=0.32\textwidth]{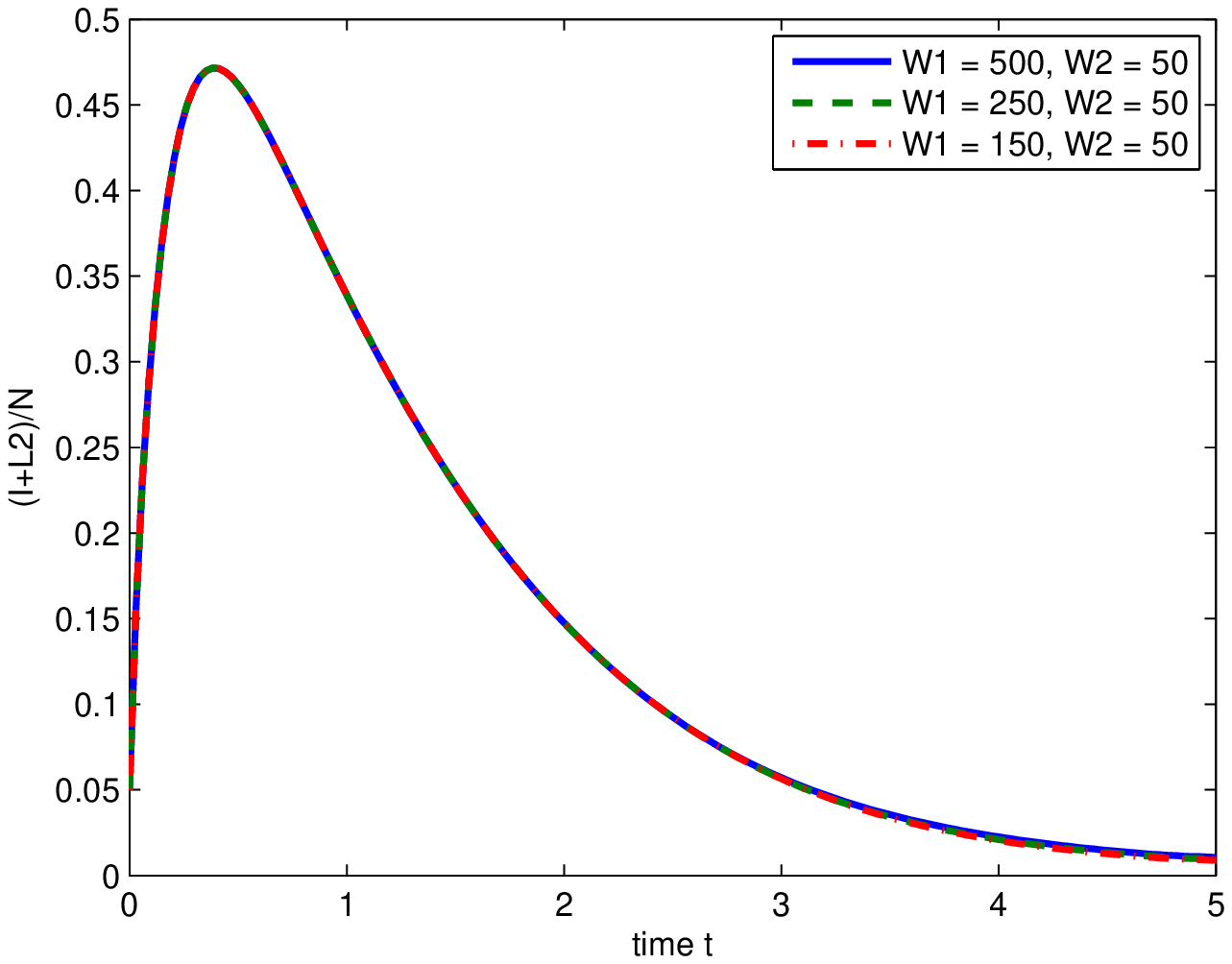}}
\caption{Fix $W_2 = 50$ and decrease $W_1$ ($W_1 = 500, 250, 150$).}
\label{fig:variar:W1}
\end{figure}

\begin{figure}[!htb]
\centering
\subfloat[\footnotesize{Control $u_1$}]{\label{fig:W1igualW2:u1}
\includegraphics[width=0.32\textwidth]{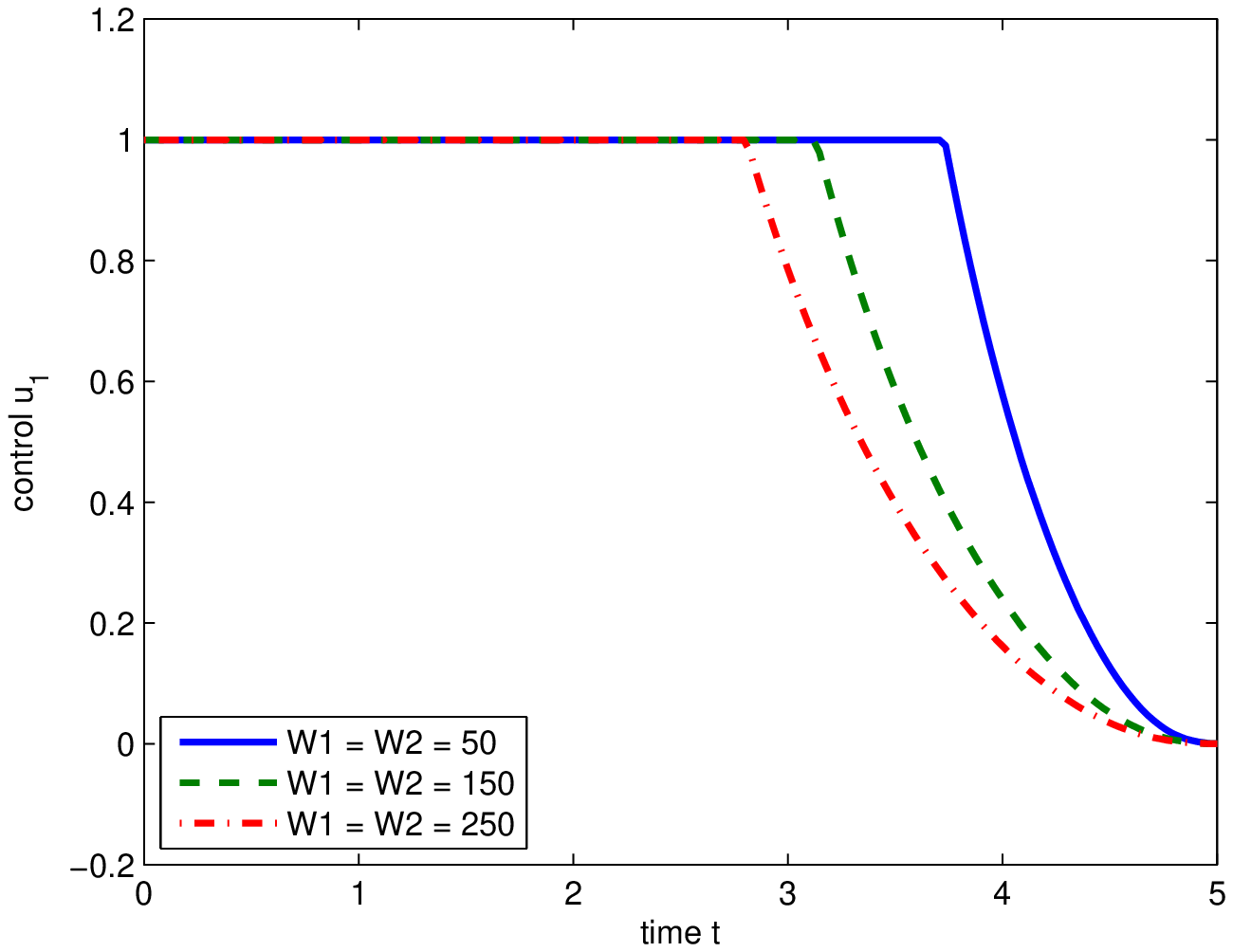}}
\subfloat[\footnotesize{Control $u_2$}]{\label{fig:W1igualW2:u2}
\includegraphics[width=0.32\textwidth]{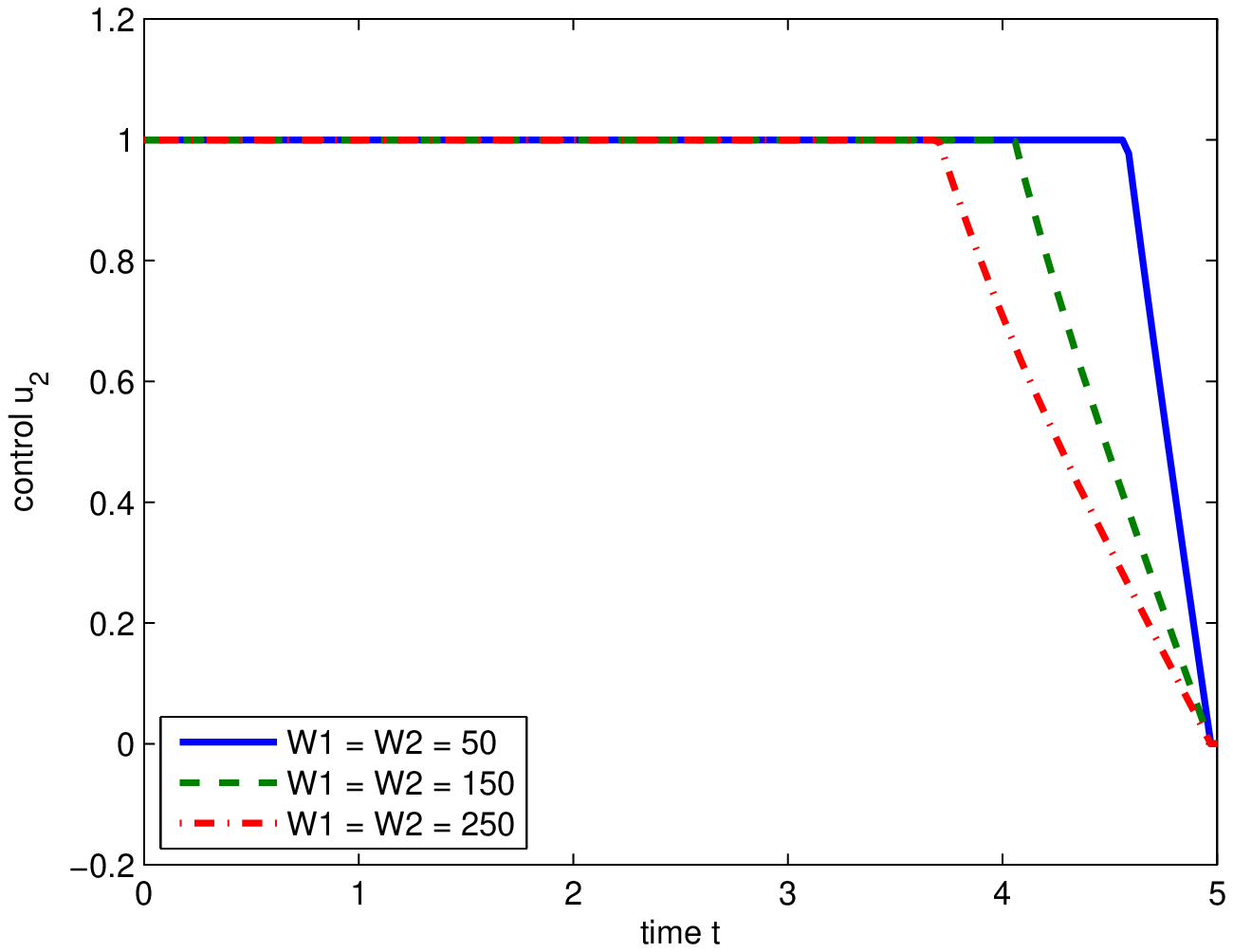}}
\subfloat[\footnotesize{$(I+L_2)/N$}]{\label{fig:W1igualW2:I:L2}
\includegraphics[width=0.32\textwidth]{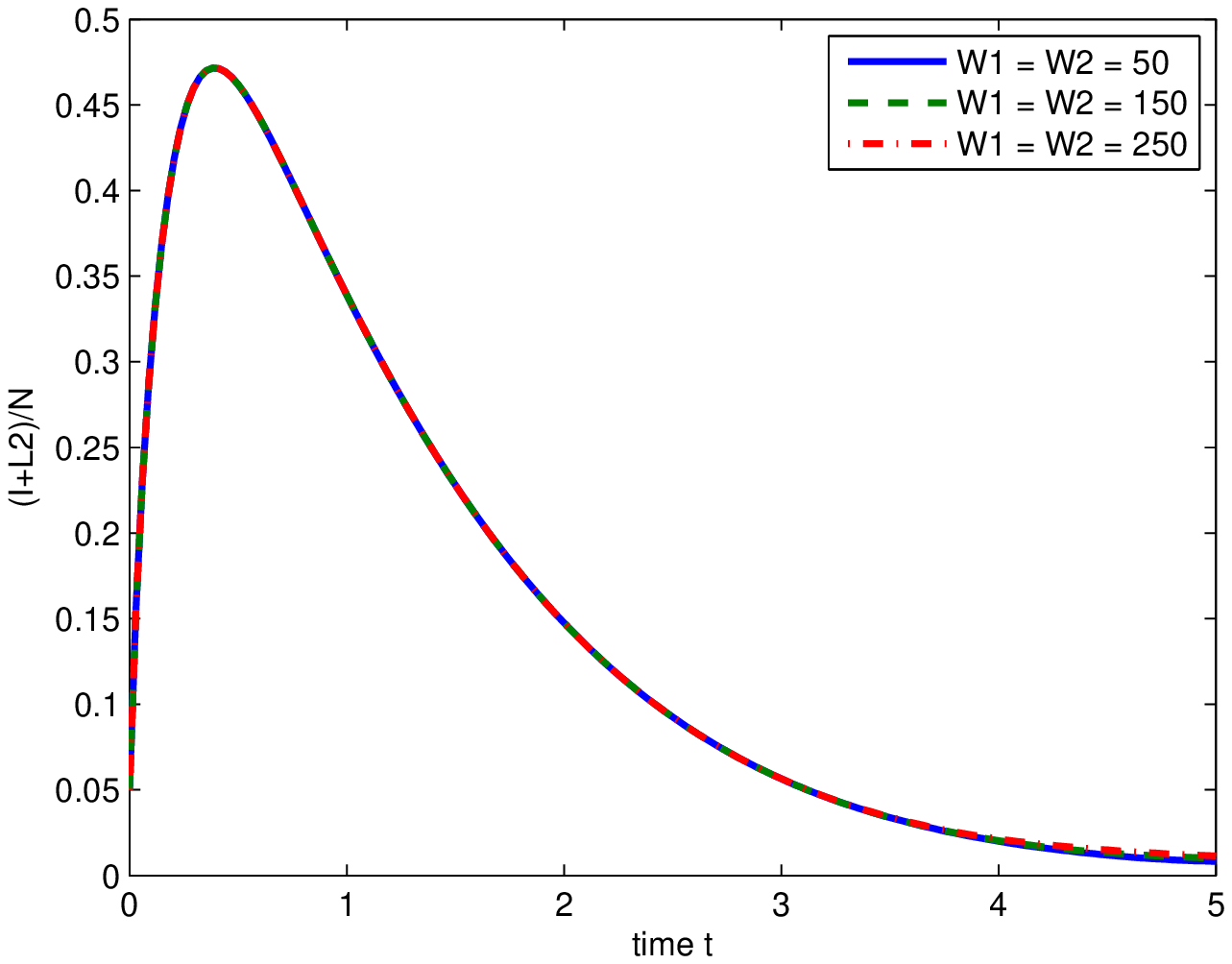}}
\caption{$W_1 = W_2= 50, 150, 250$.}
\label{fig:W1igualW2}
\end{figure}

\begin{figure}[!htb]
\centering
\subfloat[\footnotesize{Control $u_1$}]{\label{fig:variar:epsilon:u1}
\includegraphics[width=0.32\textwidth]{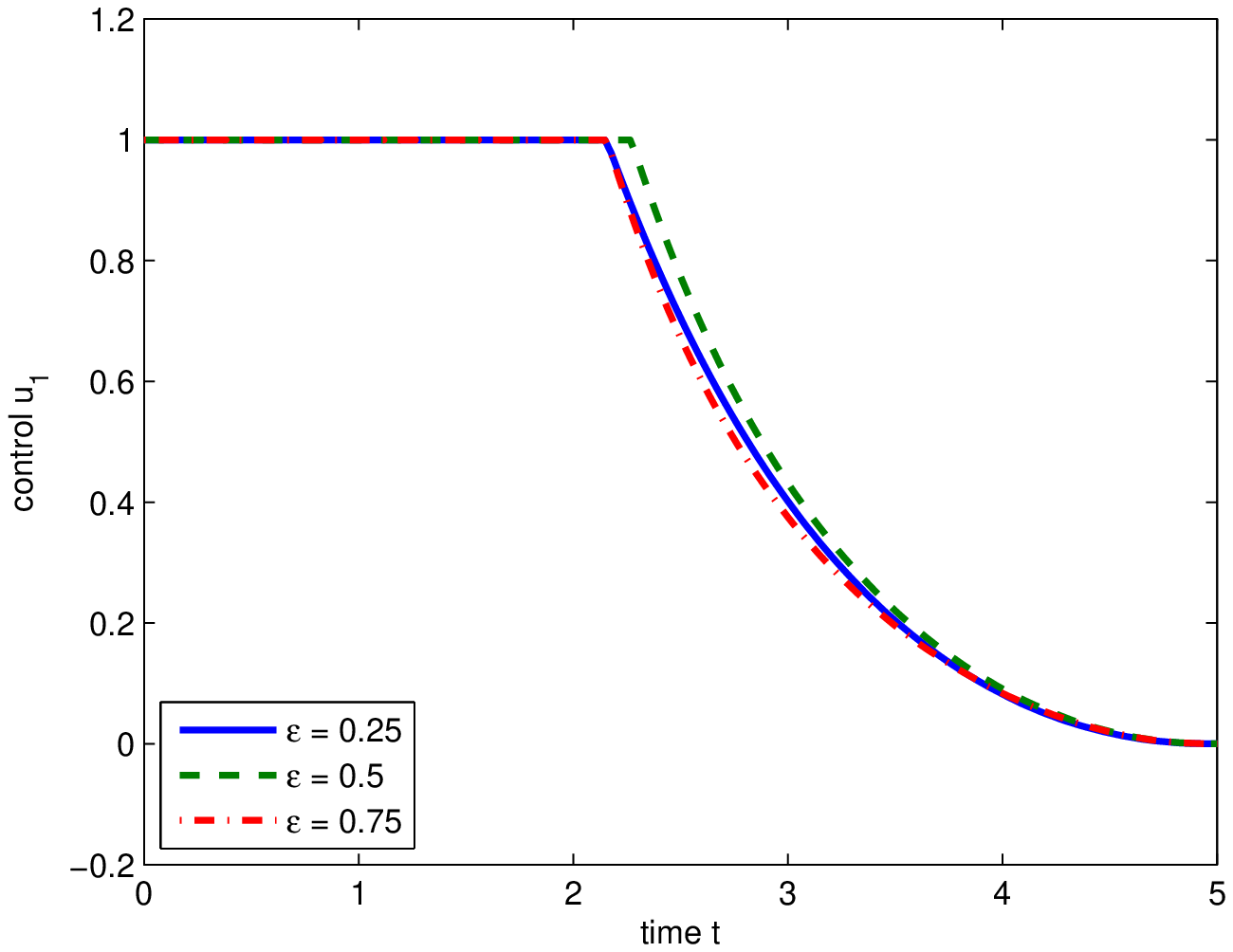}}
\subfloat[\footnotesize{Control $u_2$}]{\label{fig:variar:epsilon:u2}
\includegraphics[width=0.32\textwidth]{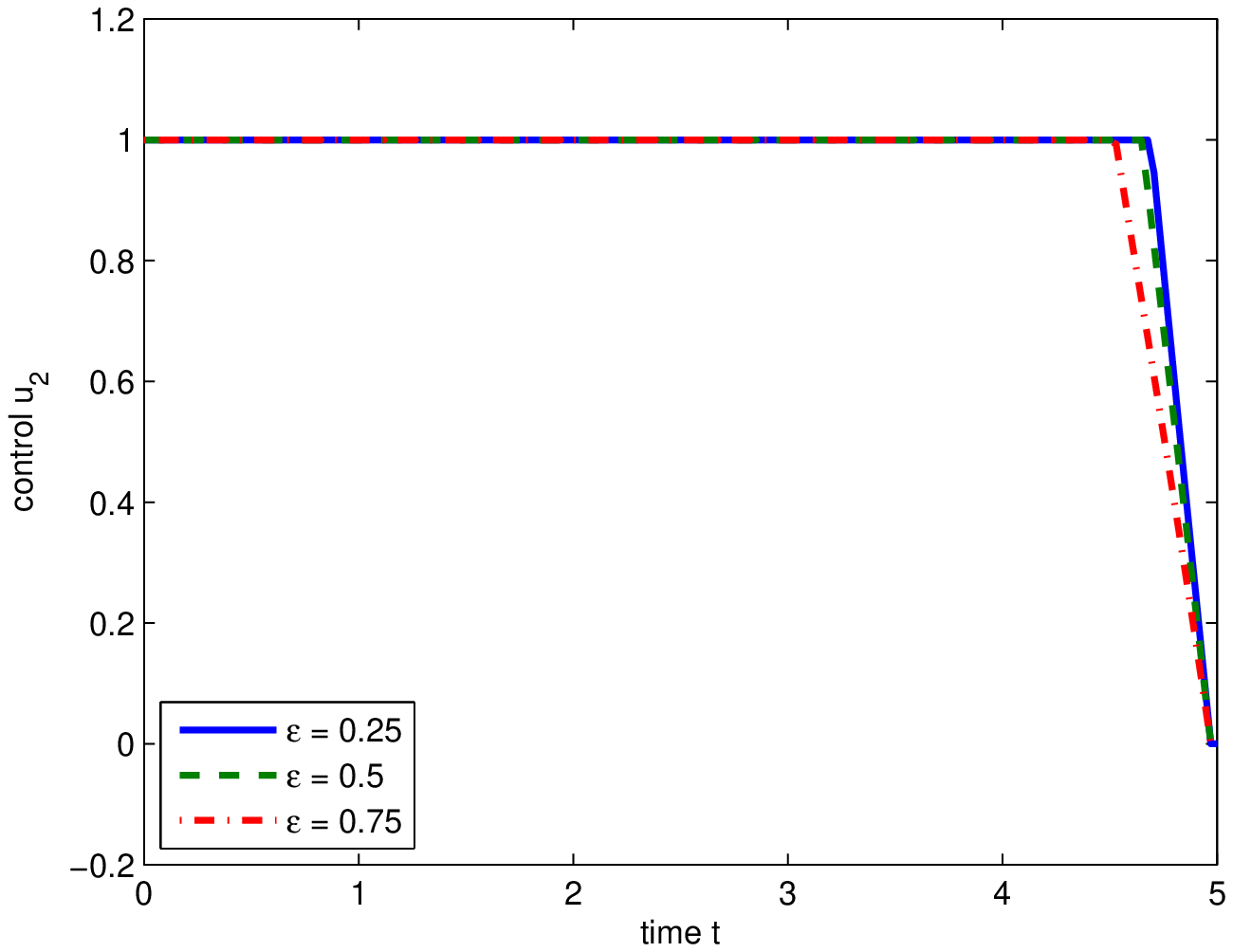}}
\subfloat[\footnotesize{$(I+L_2)/N$}]{\label{fig:variar:epsilon:I:L2}
\includegraphics[width=0.32\textwidth]{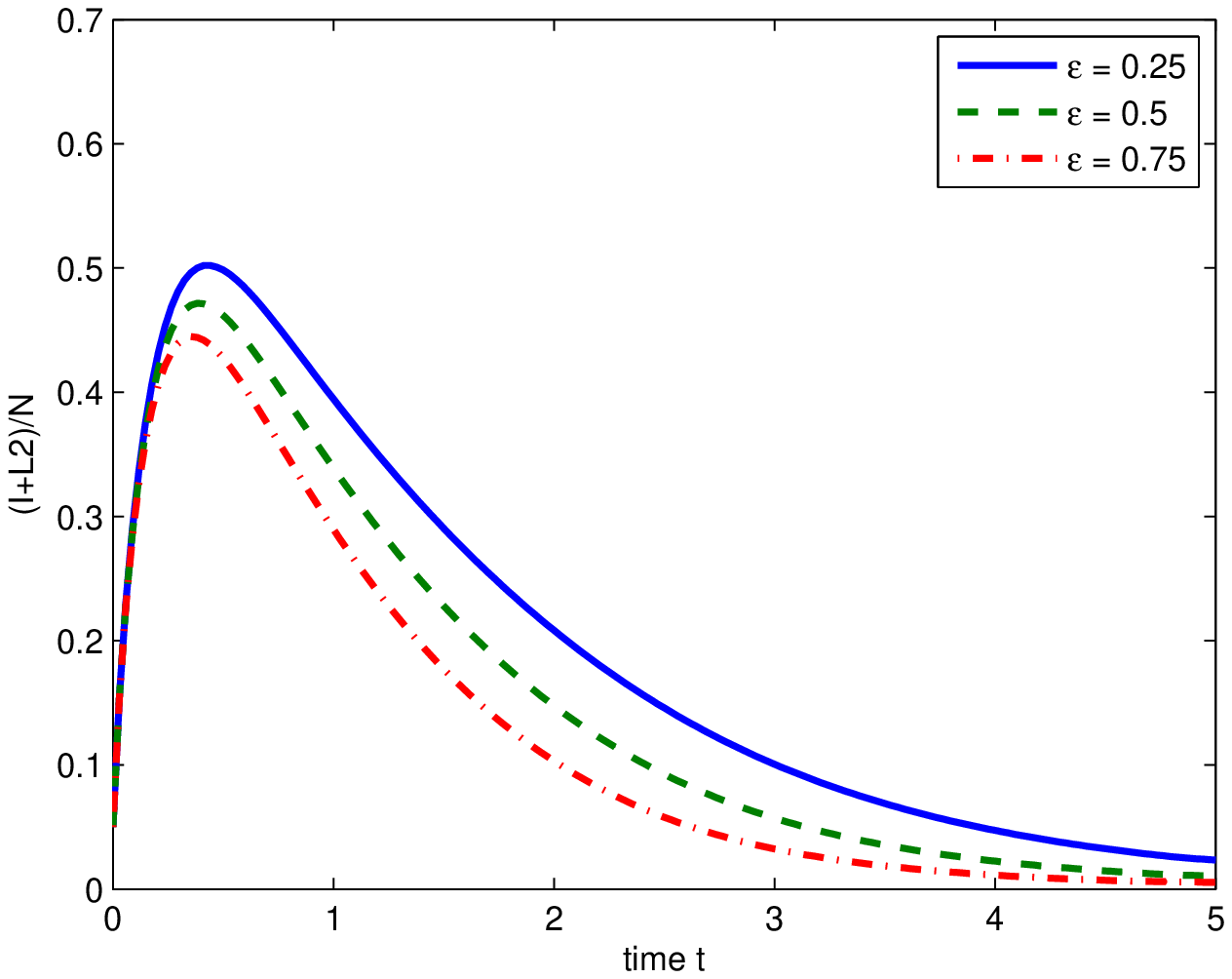}}
\caption{Variation on $\epsilon_i$, $i=1, 2$ ($\epsilon_1 = \epsilon_2 = 0.25; 0.5; 0.75$).}
\label{fig:variar:epsilon}
\end{figure}

If we increase the transmission coefficient $\beta$, a different scenario occurs.
TB rises and then declines in all of the simulations presented so far.
Because TB is clearly already endemic in most relevant settings,
we are not before the peak in Figure~\ref{fig:I:L2:com:e:sem:controlos}
and Figures~\ref{fig:variar:beta}--\ref{fig:variar:epsilon} (c).
But if we are now at the point where we are after
that peak, then TB is declining anyway in all of the figures.
However, if we take, for example, $\beta = 200$, this situation does not occur.
If we don't introduce controls, the number of active infected individuals $I$
and persistent latent individuals $L_2$ increases significantly, see Figure~\ref{fig:I:L2:semcont:b200}.
For these parameter values we have $R_0(0, 0) = 4.4$.

\begin{figure}[!htb]
\centering
\includegraphics[width=0.5\textwidth]{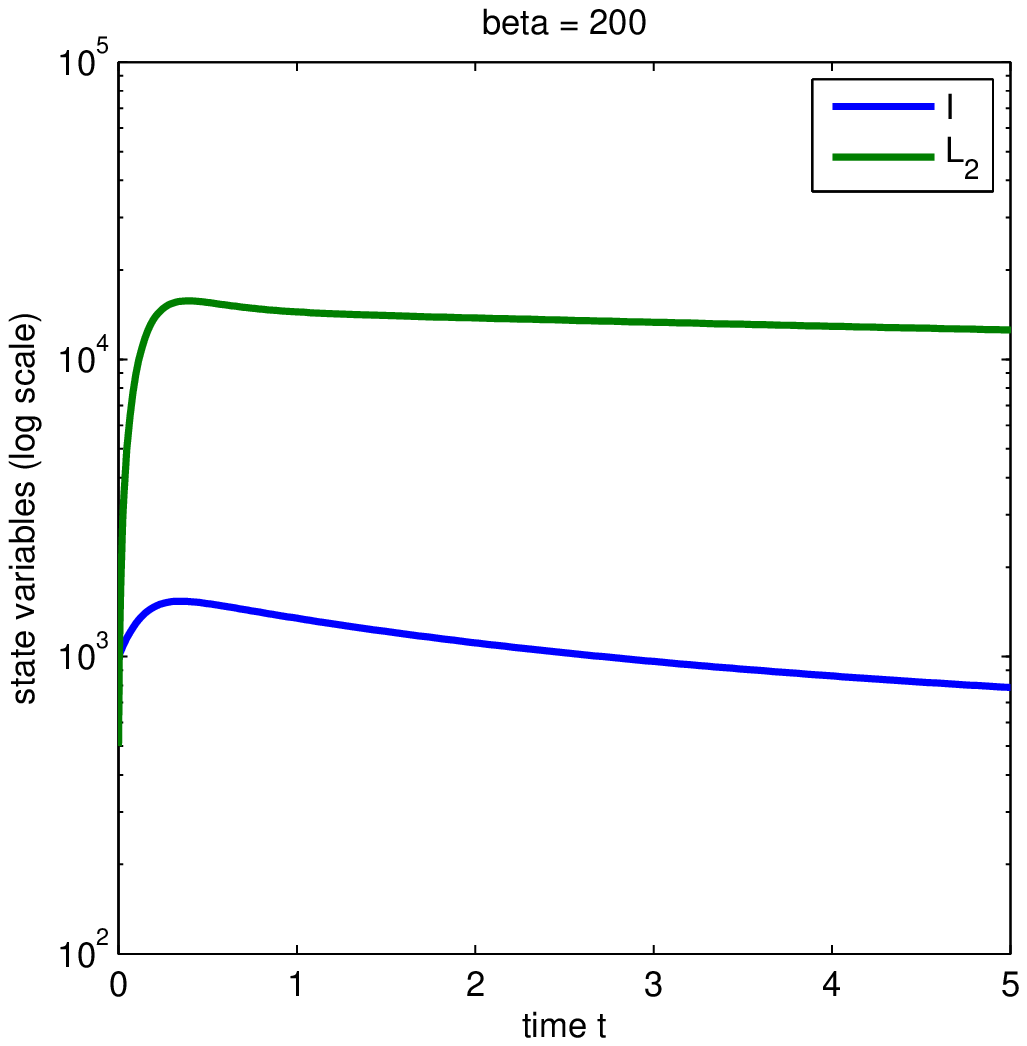}
\caption{Number of active infected individuals $I$ and persistent latent individuals $L_2$,
$\beta = 200$ and $N=30000$, when no controls are used.}
\label{fig:I:L2:semcont:b200}
\end{figure}

We now show a new situation where the controls have a crucial role:
without controls the number of infected increases, while
using the optimal control approach the number of infected decreases
and remains always in a lower level. Let $N = 30000$,
$\epsilon_1 = \epsilon_2 = 0.5$, and $W_1 = W_2 = 500$.
Take, for example, $\beta = 250$. If the control measures $u_1$ and $u_2$
are not implemented, then the number of active infected individuals $I$ increases
in all the treatment period. However, if the control measures $u_1$ and $u_2$ are used
(see Figure~\ref{fig:u_1:u_2:b250} for the optimal controls) an important
decrease of active infected individuals is observed
(see Figure~\ref{fig:I:com:e:sem:cont:b250}).
Without controls the basic reproduction number is $R_0(0, 0)=5.5$
and with controls $R_0(1, 1) = 4.4$.

\begin{figure}[!htb]
\centering
\subfloat[\footnotesize{Control $u_1$}]{\label{u_1:b250}
\includegraphics[width=0.45\textwidth]{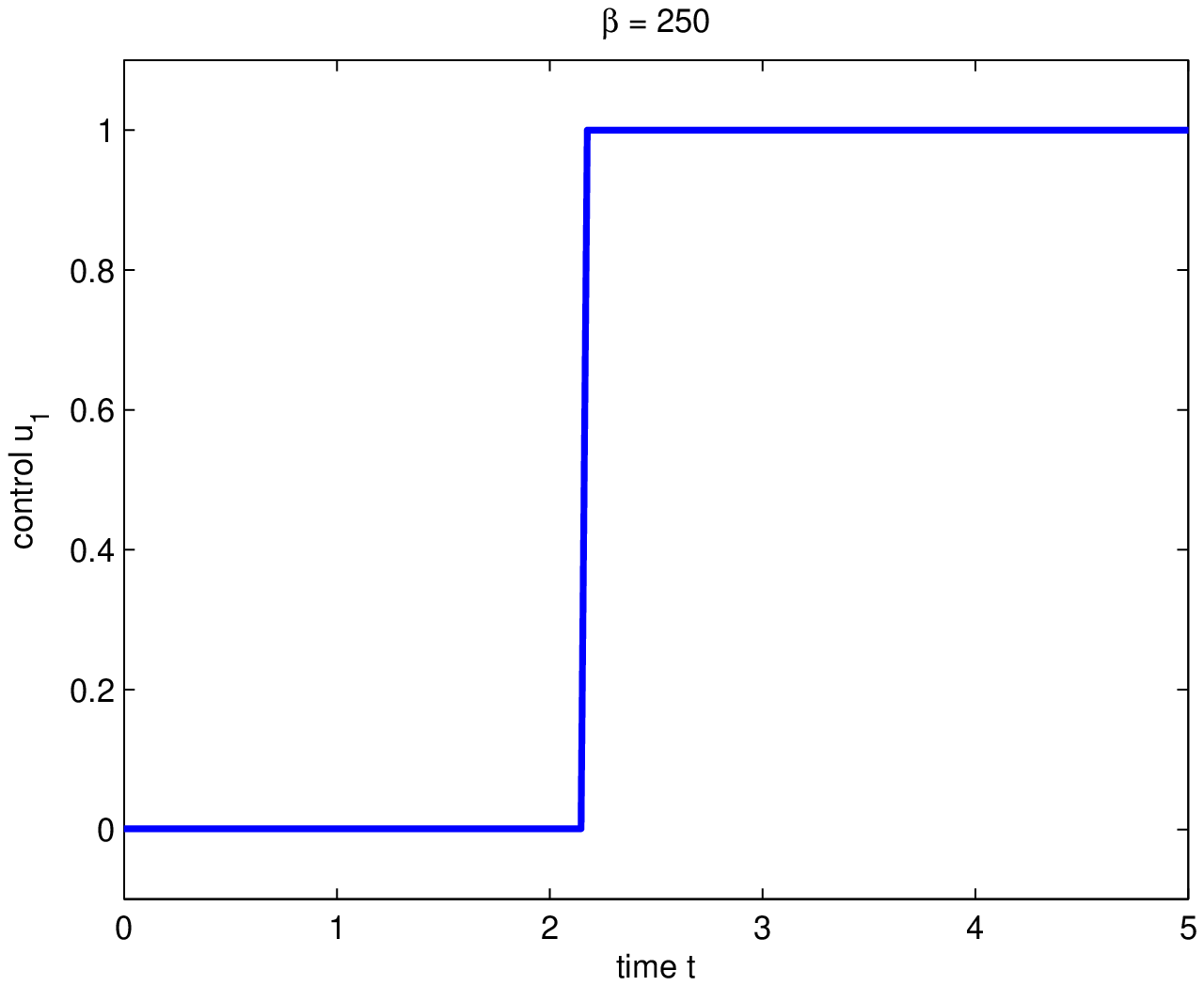}}
\subfloat[\footnotesize{Control $u_2$}]{\label{u_2:b250}
\includegraphics[width=0.45\textwidth]{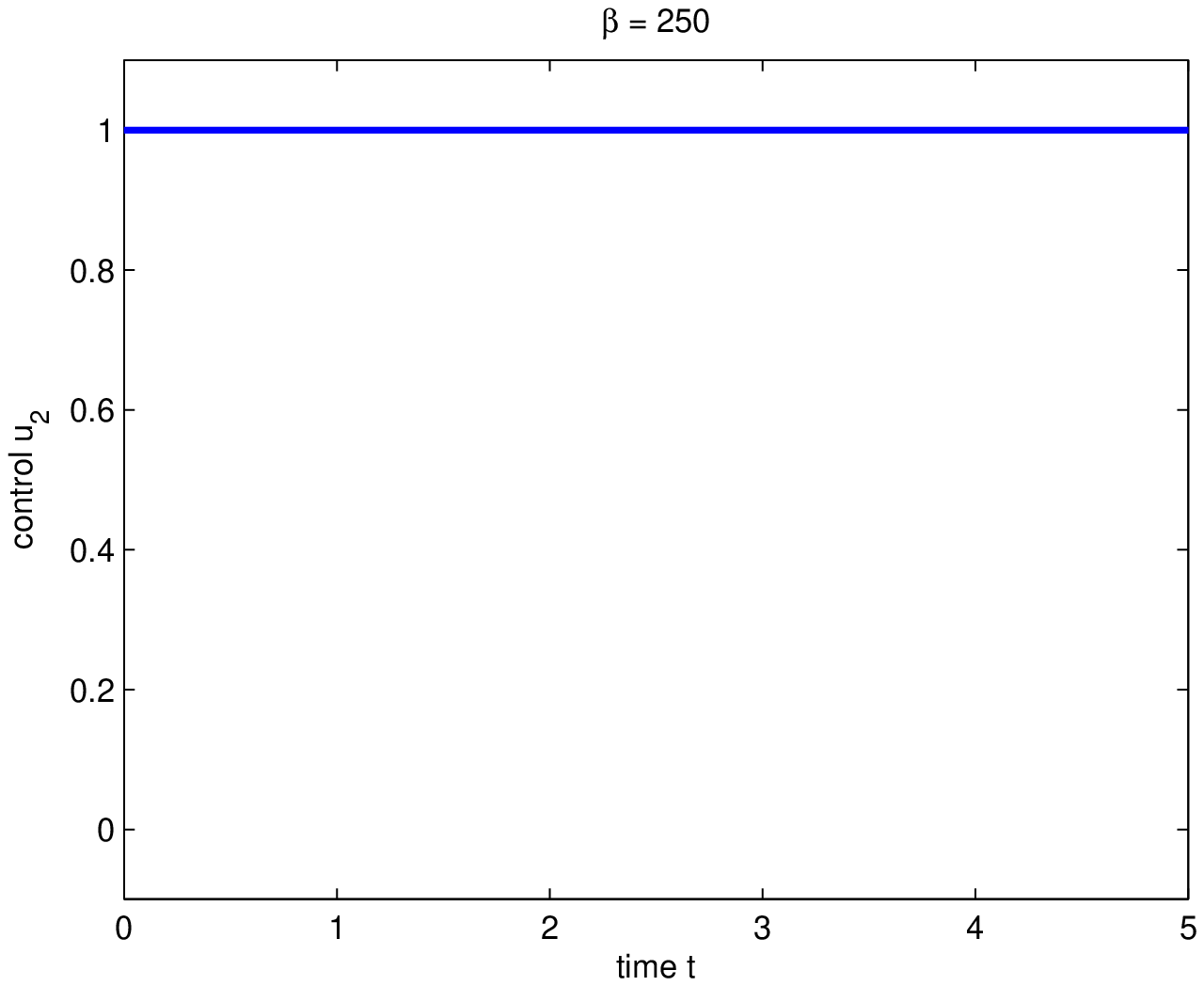}}
\caption{Controls $u_1$ and $u_2$ for $\beta = 250$;
$\epsilon_1 = \epsilon_2 = 0.5$; $W_1 = W_2 = 500$; $N=30000$.}
\label{fig:u_1:u_2:b250}
\end{figure}

\begin{figure}[!htb]
\centering
\includegraphics[width=0.5\textwidth]{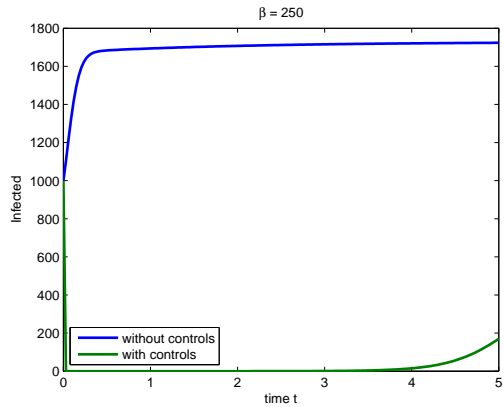}
\caption{Number of active infected individuals with and without controls.}
\label{fig:I:com:e:sem:cont:b250}
\end{figure}

Other cost functionals could be considered in the optimal control problem,
namely a cost $C$ where the category of persistent latent
individuals $L_2$ is not considered in \eqref{costfunction}, that is,
\begin{equation}
\label{eq:np:co}
C(u_1(\cdot), u_2(\cdot))
= \int_0^{T} \left[ I(t) + \frac{W_1}{2}u_1^2(t) + \frac{W_2}{2}u_2^2(t) \right] dt.
\end{equation}
However, in Figures~\ref{fig:I:L2:costs} and \ref{fig:I:L2:costs:b175}
we observe, for different values of $\beta$, that when we consider the cost
functional \eqref{costfunction}, the fraction of active infectious individuals
is lower compared with the case when we consider
only active infectious TB individuals $I$ \eqref{eq:np:co}.
\begin{figure}[!htb]
\centering
\subfloat[\footnotesize{Fraction $I/N$ of infected individuals}]{\label{I:compare:Cost}
\includegraphics[width=0.45\textwidth]{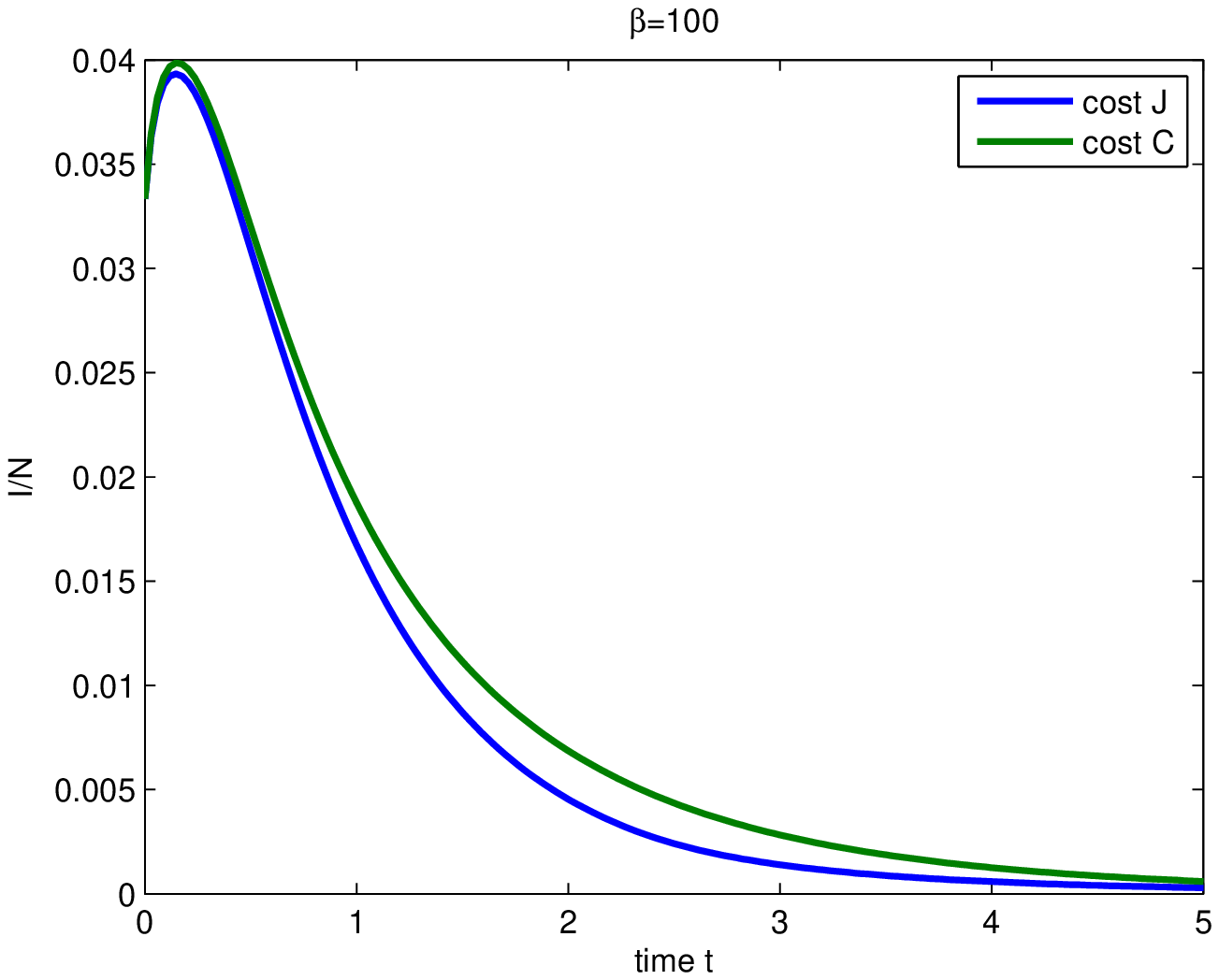}}
\subfloat[\footnotesize{Fraction $L_2/N$ of persistent latent individuals}]{\label{L2:compare:Cost}
\includegraphics[width=0.45\textwidth]{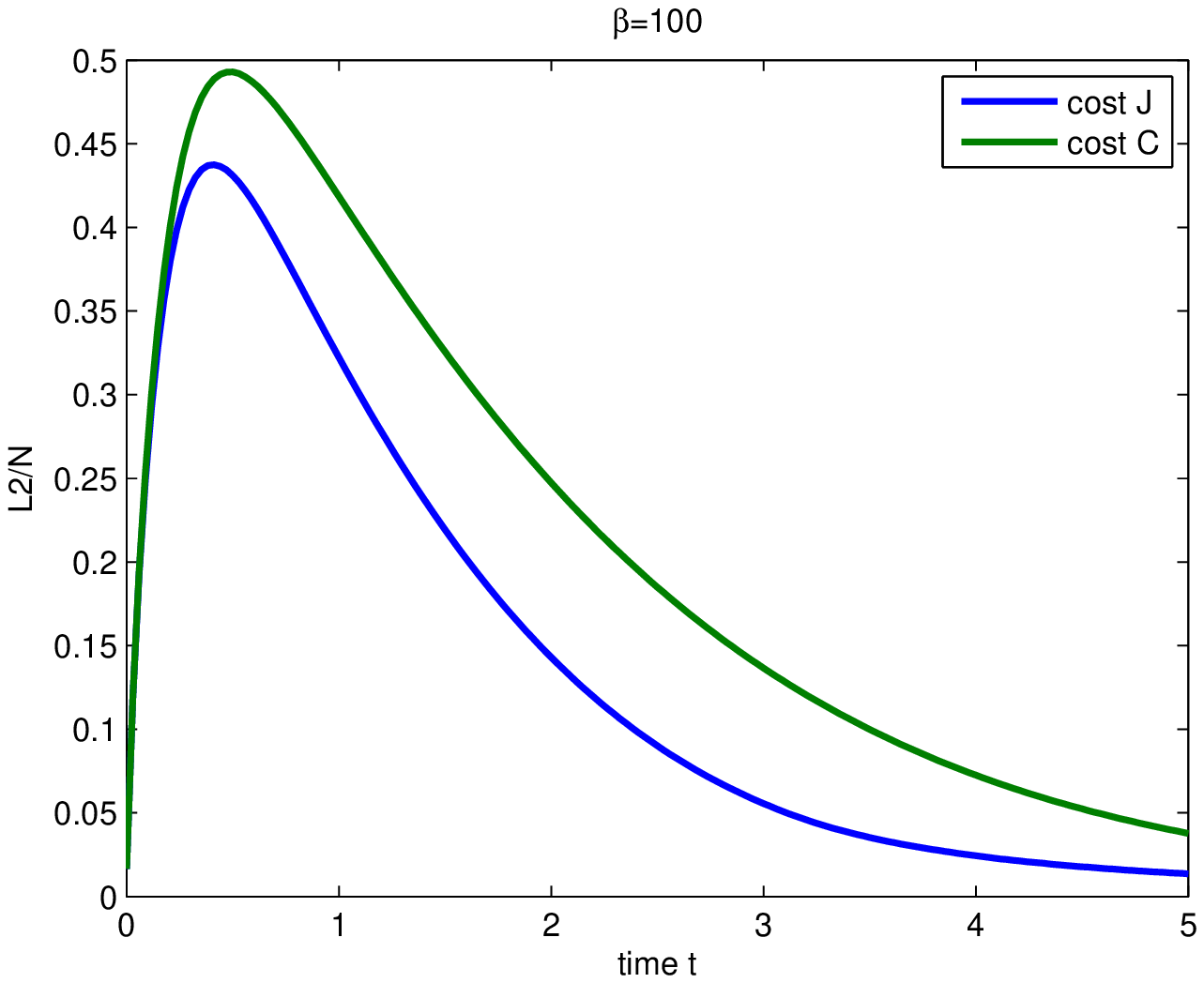}}
\caption{Comparison of active infectious $I$ and persistent latent $L_2$ individuals
for different cost functionals $J$ and $C$ ($\beta = 100$).}
\label{fig:I:L2:costs}
\end{figure}
\begin{figure}[!htb]
\centering
\subfloat[\footnotesize{Fraction $I/N$ of infected individuals}]{\label{I:compare:Cost:b175}
\includegraphics[width=0.45\textwidth]{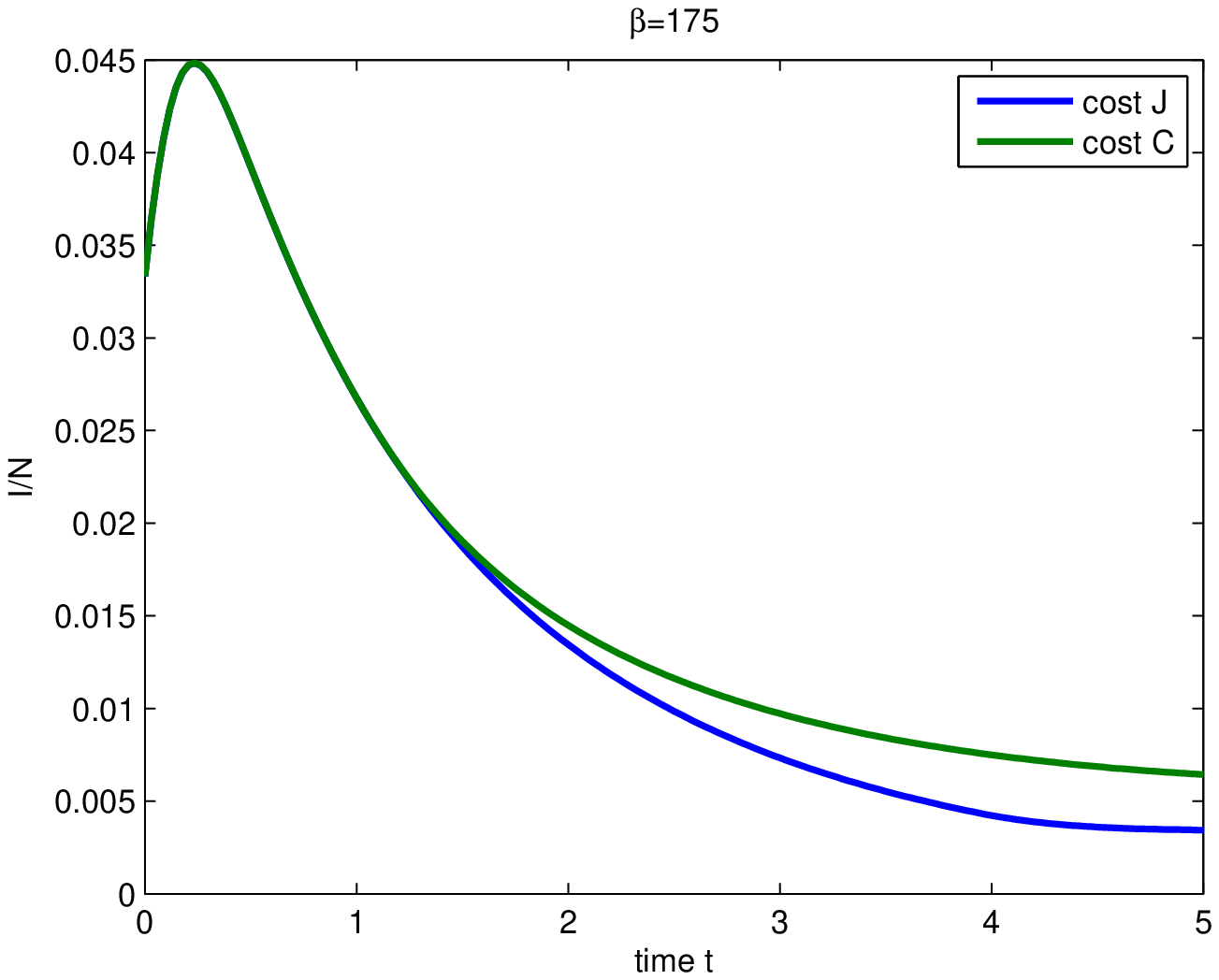}}
\subfloat[\footnotesize{Fraction $L_2/N$ of persistent latent individuals}]{\label{L2:compare:Cost:b175}
\includegraphics[width=0.45\textwidth]{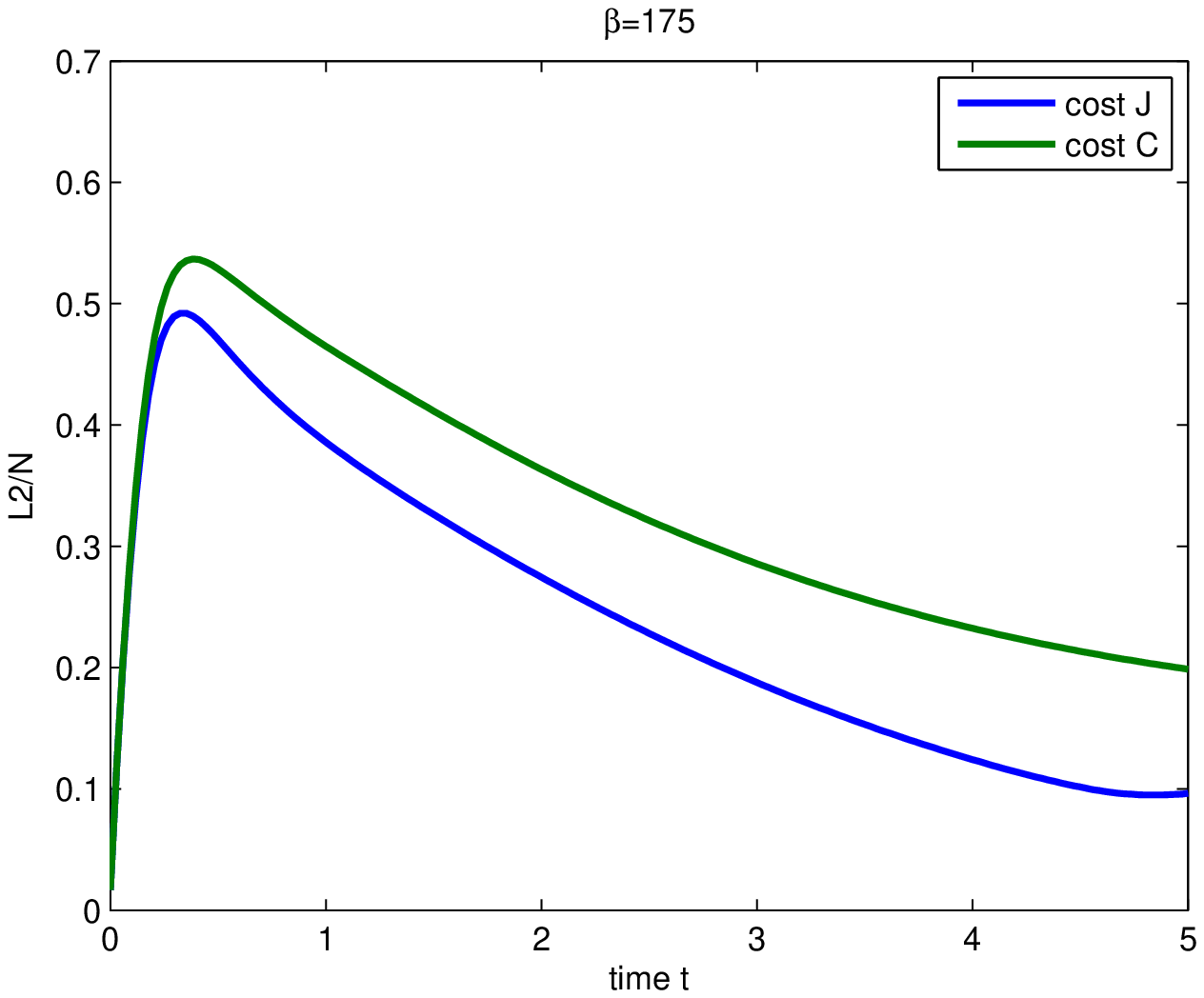}}
\caption{Comparison of active infectious $I$ and persistent latent $L_2$ individuals
for different cost functionals $J$ and $C$ ($\beta = 175$).}
\label{fig:I:L2:costs:b175}
\end{figure}

Let $N = 30000$, $\beta \in \{100, 175\}$, $\epsilon_1 = \epsilon_2 = 0.5$ and $W_1 = W_2 =500$.
When we compare the controls $u_1$ and $u_2$ for the cost functionals $J$ and $C$,
we observe that a bigger effort is required on the controls when we propose to minimize $I + L_2$.
If our aim is to minimize the active infected individuals $I$ as well as the cost
of the control measures represented by $u_1$ (preventive measures applied to active infected
individuals for a complete treatment with anti-TB drugs), then the control $u_1$ never attains
the maximum value and the control measure $u_2$ is not required, see Figure~\ref{fig:u_1:u_2:costs:b100}.

\begin{figure}[!htb]
\centering
\subfloat[\footnotesize{Control $u_1$}]{\label{u_1:compare:Cost}
\includegraphics[width=0.45\textwidth]{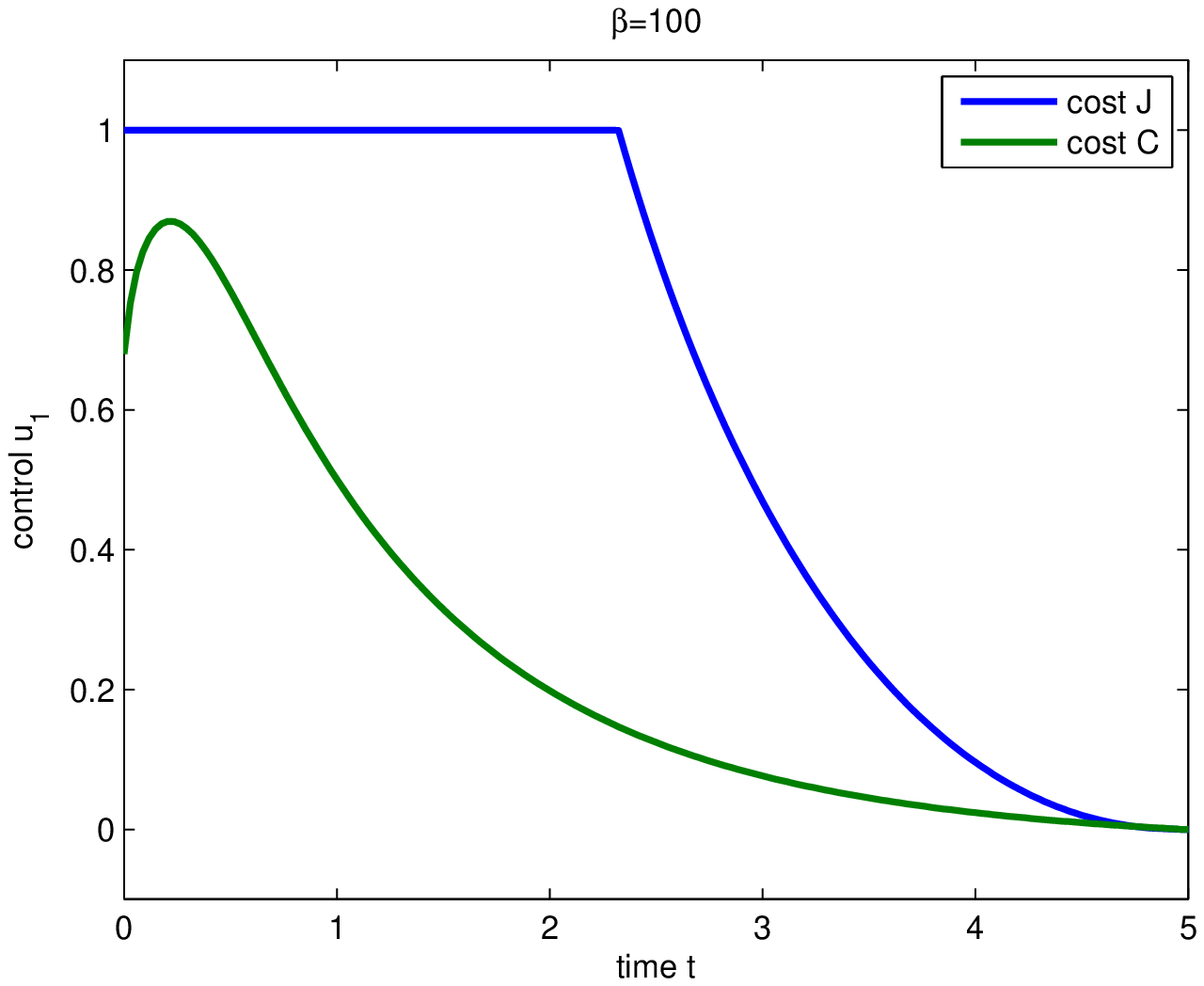}}
\subfloat[\footnotesize{Control $u_2$}]{\label{u_2:compare:Cost}
\includegraphics[width=0.45\textwidth]{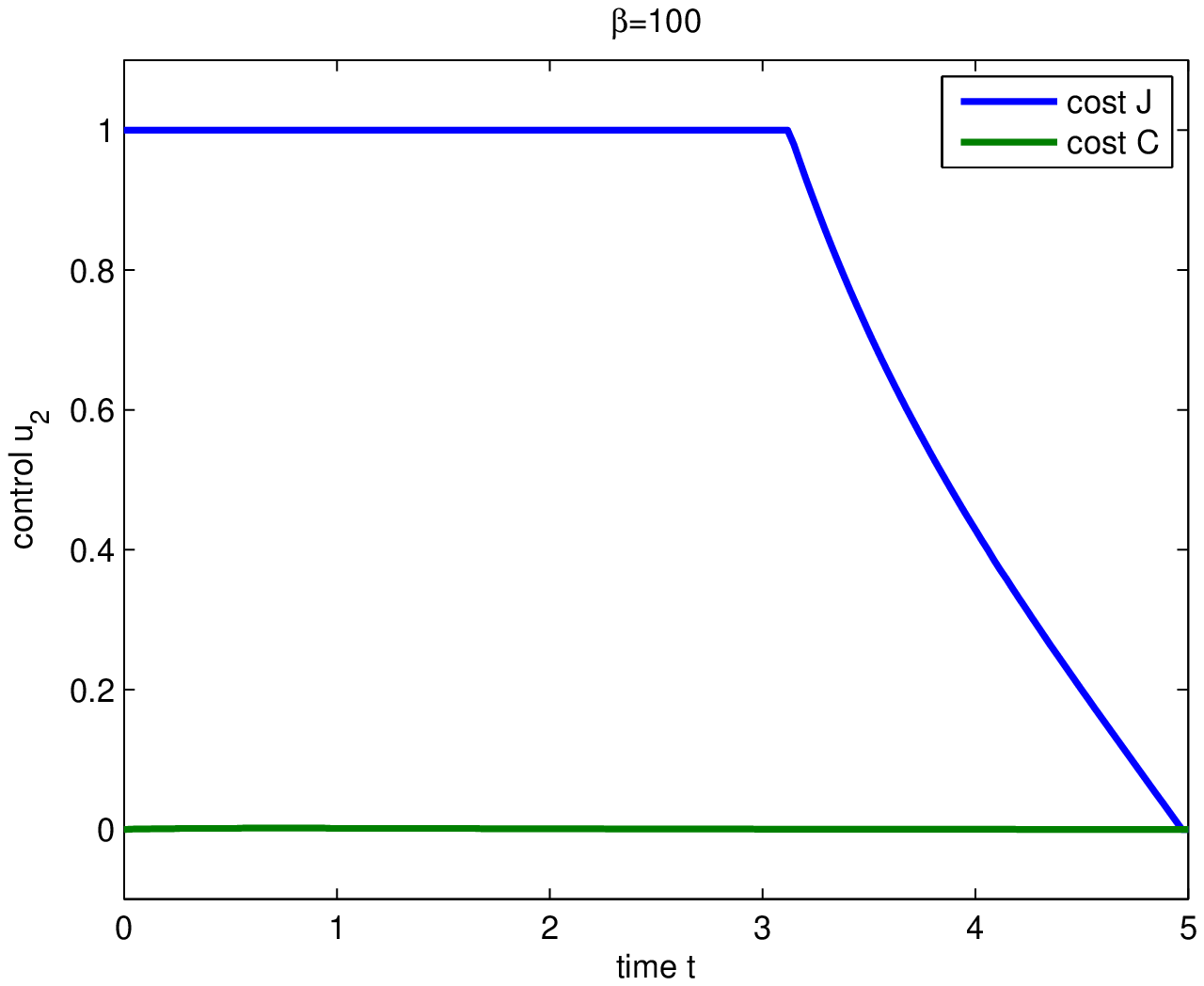}}
\caption{Comparison of controls $u_1$ and $u_2$ for different cost functionals $J$ and $C$ ($\beta = 100$).}
\label{fig:u_1:u_2:costs:b100}
\end{figure}

When $\beta = 175$, if we want to minimize the cost functional $C$
then the control $u_1$ attains the maximum value for approximately 1.1 years;
if we want to minimize the cost functional $J$
then the control $u_1$ attains the maximum value for approximately 3.9 years.
Analogously to the case $\beta = 100$, for $\beta = 175$
the control measure $u_2$ is not required when we wish
to minimize $C$, see Figure~\ref{fig:u_1:u_2:costs:b175}.

\begin{figure}[!htb]
\centering
\subfloat[\footnotesize{Control $u_1$}]{\label{u_1:compare:Cost:b175}
\includegraphics[width=0.45\textwidth]{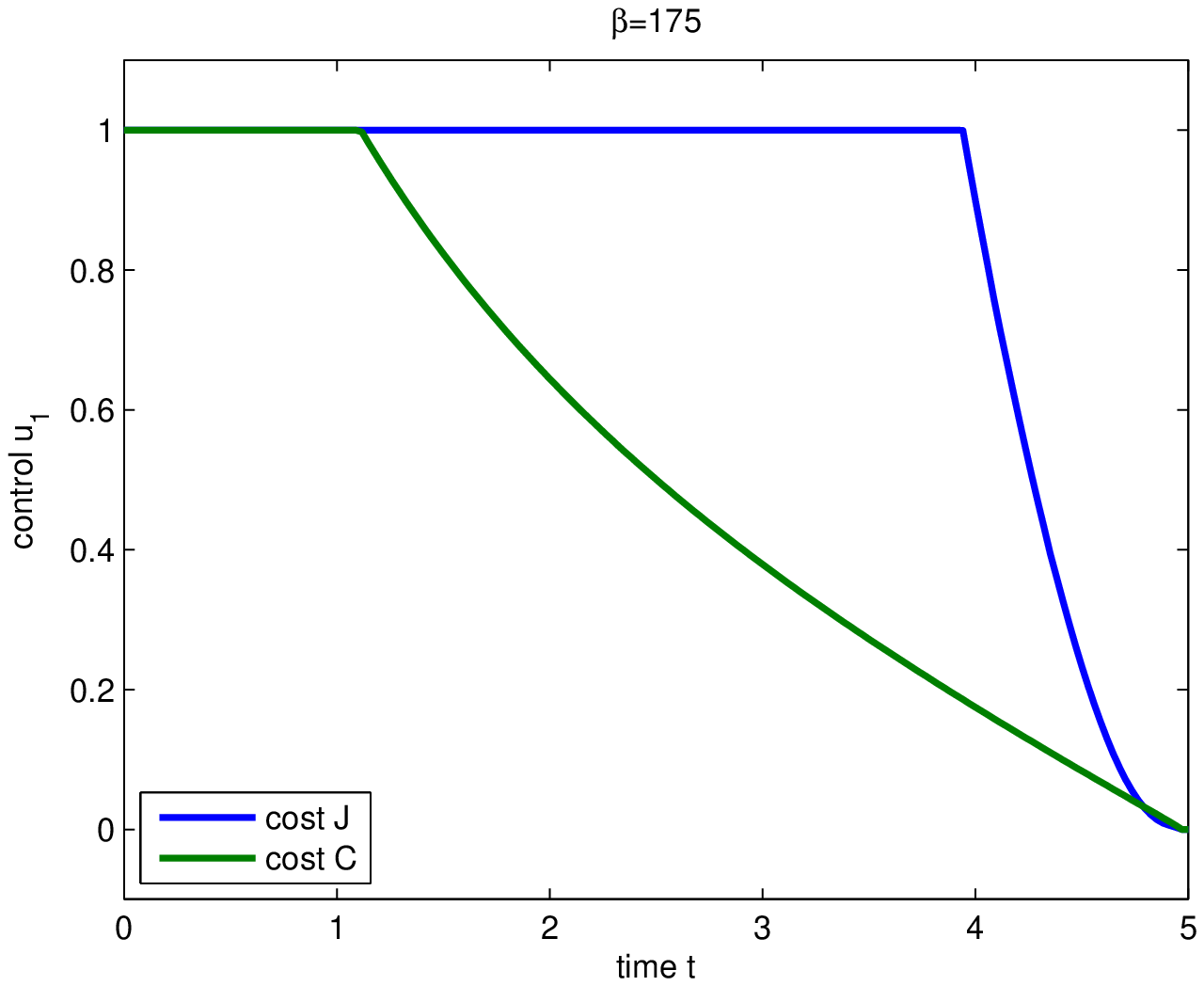}}
\subfloat[\footnotesize{Control $u_2$}]{\label{u_2:compare:Cost:b175}
\includegraphics[width=0.45\textwidth]{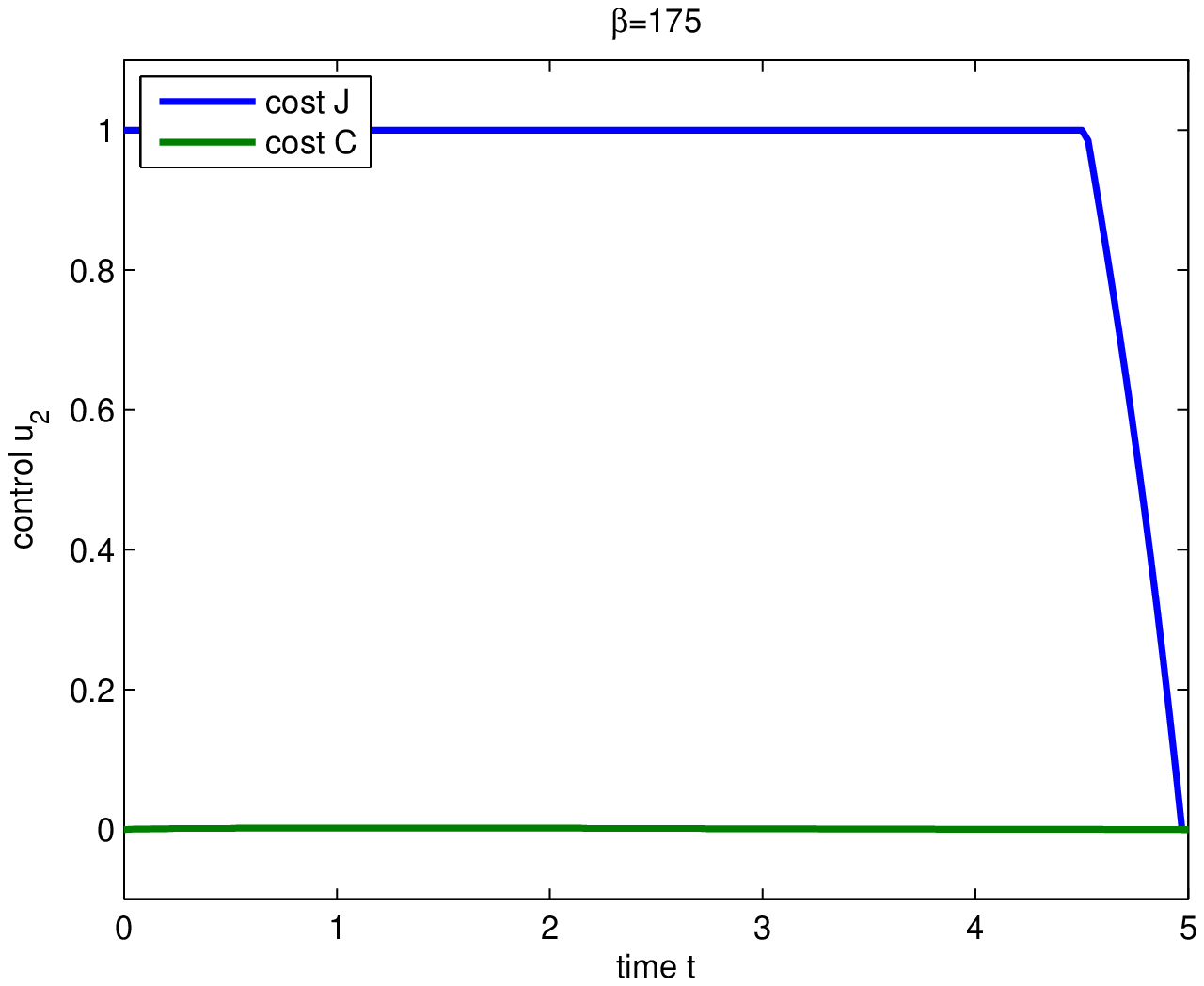}}
\caption{Comparison of controls $u_1$ and $u_2$ for different cost functionals $J$ and $C$ ($\beta = 175$).}
\label{fig:u_1:u_2:costs:b175}
\end{figure}
The bigger effort on the control measures associated to the cost functional $J$ is justified by
a reduction on the fraction of active infected individuals $I$ and persistent latent individuals
$L_2$, when compared to the case of minimizing $C$, see Figure~\ref{fig:I:L2:costs}.
This reduction is more significant if we consider $\beta = 175$, see Figure~\ref{fig:I:L2:costs:b175}.

Finally, we can compare the effect of the implementation of two controls strategies:
apply only the control $u_1$ (supervision and support of TB infectious individuals $I$);
and apply simultaneously $u_1$ and $u_2$. The treatment of persistent latent individuals $L_2$,
for example with a prophylactic vaccine, is not so usual as the treatment of active infectious individuals,
which is one of the measures proposed by the Direct Observation Therapy (DOT) of World Health Organization (WHO),
but is a valid TB treatment strategy \cite{Abu_Raddad:etal:2009,Cohen:etal:2006}.
Observing Figure~\ref{fig:I:L2:only:u1}, we conclude that each control $u_1$ and $u_2$ implies
a reduction on the respective fraction of the population, $I$ and $L_2$. Moreover,
if we choose to minimize the cost function \eqref{costfunction},
then the best choice is to apply both controls simultaneously,
since the implementation of control $u_2$ does not imply a reduction on the fraction
of active infectious individuals. For this reason,
if we choose to minimize the cost function \eqref{eq:np:co},
then the best control strategy is to implement only control $u_1$.

\begin{figure}[!htb]
\centering
\subfloat[\footnotesize{Fraction $I/N$ of infected individuals}]{\label{u1:b250}
\includegraphics[width=0.45\textwidth]{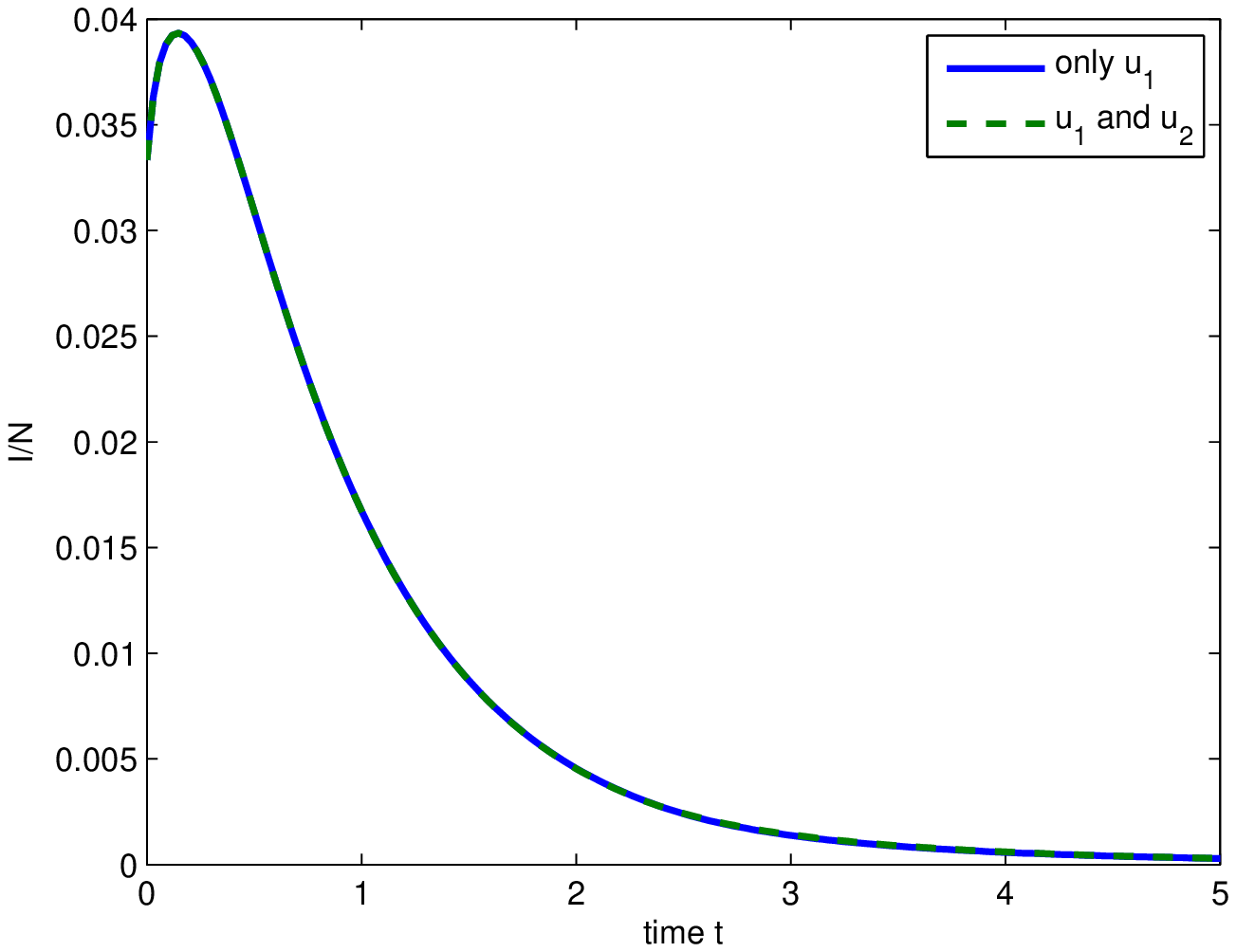}}
\subfloat[\footnotesize{Fraction $L_2/N$ of persistent latent individuals}]{\label{u2:b250}
\includegraphics[width=0.45\textwidth]{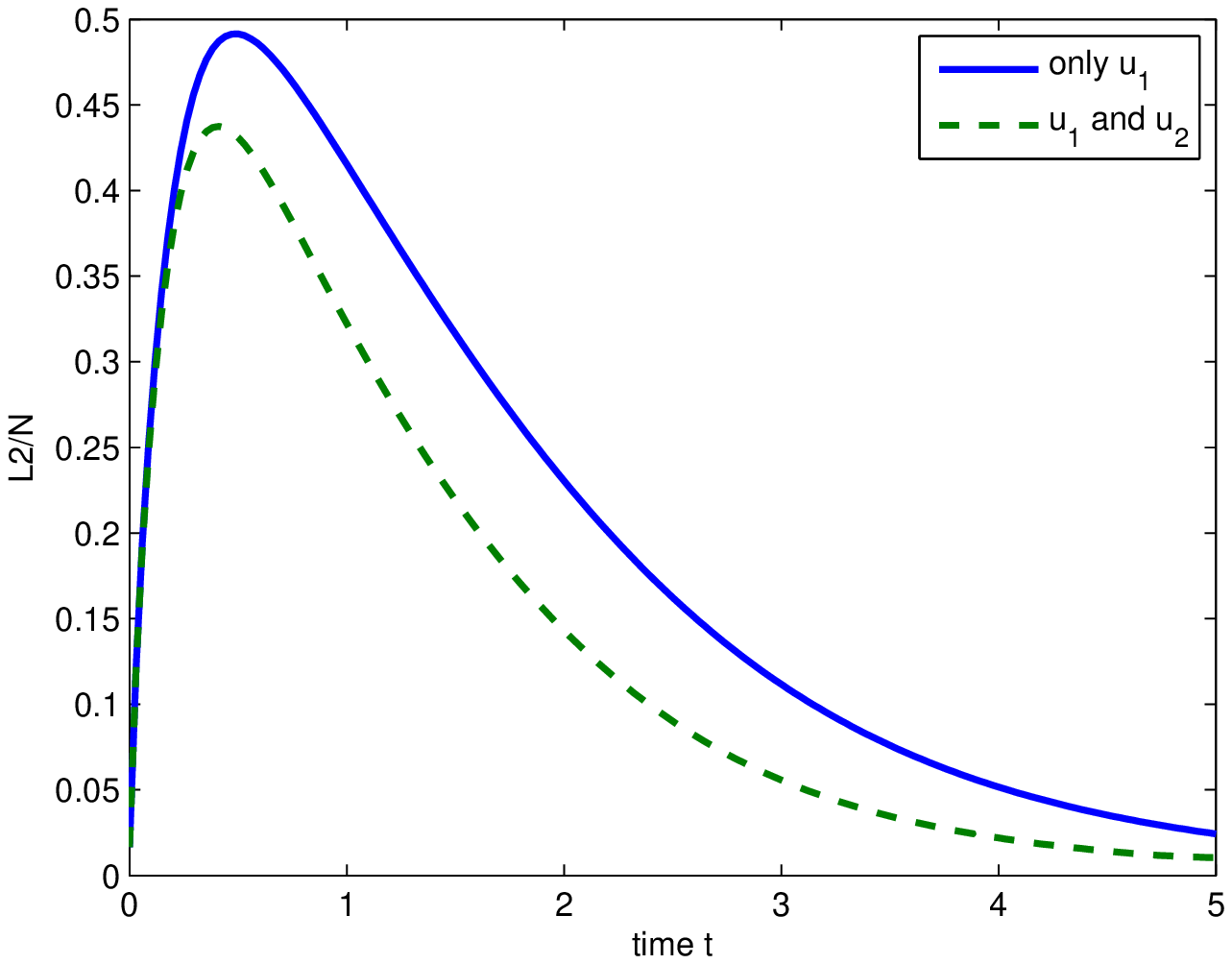}}
\caption{Comparison between applying only control $u_1$
and applying the two controls $u_1$ and $u_2$ simultaneously, with
$\beta = 100$, $\epsilon_1 = \epsilon_2 = 0.5$, $W_1= 500$, $W_2 = 50$, and $N=30000$.}
\label{fig:I:L2:only:u1}
\end{figure}


\section{Conclusion}
\label{sec:conc}

The incidence rates of TB have been declining since 2004 worldwide,
namely due to prevention and treatment policies
that have been applied in the last years \cite{WHO_2011}.
Mortality rates, at global level, fell down around 35\% between 1990 and 2009,
and if the current rate of decline is sustained, by 2015 the target of a 50\%
reduction can be achieved. The reduction of mortality and incidence rates depend
on the effort at country level to implement control policies.
Examples of policies that had a great success in WHO's six regions
are the DOTS strategy (1995-2005) and its sucessor
\emph{Stop TB} launched in 2006 (see \cite{WHO_2011} for further details).
In this paper we study a mathematical model for TB proposed in \cite{Gomes_etall_2007},
from the optimal control point of view. Optimal time-dependent prevention
policies, that consider the execution cost, are proposed.
We tried different numerical approaches and we observed that the results are the same,
independently of the method used. In particular, two approaches have been tried:
direct and indirect. The direct methods discretize the problem turning it
into a nonlinear optimization problem. Indirect methods use the Pontryagin Maximum Principle
\cite{Pontryagin_et_all_1962} as a necessary condition
to find the optimal curve for the respective control:
we substitute \eqref{optcontrols} into \eqref{adjoint_function} and
the obtained system of five equations is solved numerically
together with the five equations of system \eqref{modelGab_controls}.
Figures~\ref{fig:com:e:sem:controlos}
and \ref{fig:u1u2:com:e:sem:controlos} show that when controls are considered,
the optimal policies provide a reduction of 1230 active infectious and persistent latent individuals.
The cost execution of the control policies related to control $u_1$ is assumed to be greater
or equal to the one related to control $u_2$ (see Section~\ref{sec:num:results}).
Figures~\ref{fig:variar:W1} and \ref{fig:W1igualW2} show that when the cost of implementation
of control policies related to control $u_1$ decreases, the amount of $u_1$ increases,
and when the cost of implementation of control policies related to control $u_2$ decreases,
the amount of $u_2$ increases. We considered different values for the transmission
coefficient parameter $\beta$ corresponding to the case where the disease
may become endemic. We observe that as the transmission coefficient increases,
the period of time that the control $u_1$ (associated to the effort that prevents
the failure of treatment of active infectious individuals) is at the upper bound
also increases, as well as the fraction of active infectious and persistent
latent individuals (Figure~\ref{fig:variar:beta}). We assume that the total population
$N$ is constant and Figure~\ref{fig:variar:N} illustrates that optimal control strategies
do not vary significantly when different sizes of population are taken.
As we can see in Figure~\ref{fig:variar:epsilon}, the measures of the efficacy
of control policies $\epsilon_1$ and $\epsilon_2$ strongly influence the effect
of the control policies related to $u_1$ and $u_2$ on the minimization
of the number of active infectious and persistent latent individuals.

As future work, it would be interesting to consider different values for the parameters
$\tau_1$ and $\tau_2$, and observe the variations on the optimal control strategies.
In addition, we intend to study optimal control strategies for the minimization
of the fraction of active infectious and/or persistent latent individuals,
when susceptibility to reinfection of treated individuals differs from that of latent:
$\sigma_R < \sigma$ or $\sigma_R > \sigma$. Another direction of research consists
to study TB/HIV co-infections.


\section*{Acknowledgments}

Work supported by {\it FEDER} funds through
{\it COMPETE} --- Operational Programme Factors of Competitiveness
(``Programa Operacional Factores de Competitividade'')
and by Portuguese funds through the
{\it Center for Research and Development
in Mathematics and Applications} (University of Aveiro)
and the Portuguese Foundation for Science and Technology
(``FCT --- Funda\c{c}\~{a}o para a Ci\^{e}ncia e a Tecnologia''),
within project PEst-C/MAT/UI4106/2011
with COMPETE number FCOMP-01-0124-FEDER-022690.
Silva was also supported by FCT through
the post-doc fellowship SFRH/BPD/72061/2010;
Torres by the OCHERA project PTDC/EEI-AUT/1450/2012.

The authors are very grateful to two anonymous referees,
for valuable remarks and comments, which
significantly contributed to the quality of the paper.



\end{document}